\documentclass[10pt]{amsart}
\usepackage{amssymb}
\usepackage{amscd}
\usepackage[all]{xy}

\numberwithin{equation}{section}

\def\today{\number\day\space\ifcase\month\or   January\or February\or
   March\or April\or May\or June\or   July\or August\or September\or
   October\or November\or December\fi\   \number\year}

\theoremstyle{definition}
\newtheorem{thm}{Theorem}[section]
\newtheorem{lem}[thm]{Lemma}
\newtheorem{prp}[thm]{Proposition}
\newtheorem{dfn}[thm]{Definition}
\newtheorem{cor}[thm]{Corollary}

\newtheorem{ntn}[thm]{Notation}
\newtheorem{exa}[thm]{Example}

\newtheorem{qst}[thm]{Question}

\newcommand{\beq}{\begin{equation}}
\newcommand{\eeq}{\end{equation}}
\newcommand{\beqr}{\begin{eqnarray*}}
\newcommand{\eeqr}{\end{eqnarray*}}
\newcommand{\bal}{\begin{align*}}
\newcommand{\eal}{\end{align*}}
\newcommand{\bei}{\begin{itemize}}
\newcommand{\eei}{\end{itemize}}

\newcommand{\af}{\alpha}
\newcommand{\bt}{\beta}

\newcommand{\ep}{\varepsilon}

\newcommand{\ch}{\chi}
\newcommand{\io}{\iota}

\newcommand{\ld}{\lambda}
\newcommand{\sm}{\sigma}
\newcommand{\kp}{\kappa}
\newcommand{\ph}{\varphi}
\newcommand{\ps}{\psi}
\newcommand{\rh}{\rho}
\newcommand{\om}{\omega}

\newcommand{\Ld}{\Lambda}

\newcommand{\Z}{{\mathbb{Z}}}
\newcommand{\R}{{\mathbb{R}}}
\newcommand{\C}{{\mathbb{C}}}
\newcommand{\N}{{\mathbb{Z}}_{> 0}}
\newcommand{\Nz}{{\mathbb{Z}}_{\geq 0}}

\newcommand{\CC}{{\mathcal{C}}}
\newcommand{\CP}{{\mathcal{P}}}

\newcommand{\RR}{{\operatorname{RR}}}

\newcommand{\id}{{\operatorname{id}}}
\newcommand{\ev}{{\operatorname{ev}}}

\newcommand{\dist}{{\operatorname{dist}}}

\newcommand{\spec}{{\operatorname{sp}}}
\newcommand{\Prim}{{\operatorname{Prim}}}

\newcommand{\supp}{{\operatorname{supp}}}

\newcommand{\card}{{\operatorname{card}}}
\newcommand{\Aut}{{\operatorname{Aut}}}

\newcommand{\Ker}{{\operatorname{Ker}}}

\newcommand{\dirlim}{\varinjlim}

\newcommand{\andeqn}{\,\,\,\,\,\, {\mbox{and}} \,\,\,\,\,\,}

\newcommand{\tfae}{the following are equivalent}
\newcommand{\ifo}{if and only if}

\newcommand{\uca}{unital C*-algebra}

\newcommand{\hm}{homomorphism}

\newcommand{\hsa}{hereditary subalgebra}

\newcommand{\pj}{projection}

\newcommand{\nzp}{nonzero projection}
\newcommand{\mvnt}{Murray-von Neumann equivalent}

\newcommand{\ct}{continuous}
\newcommand{\cfn}{continuous function}

\newcommand{\chs}{compact Hausdorff space}
\newcommand{\hme}{homeomorphism}

\newcommand{\ov}{\overline}
\newcommand{\I}{\infty}
\newcommand{\E}{\varnothing}

\newcommand{\OT}{{\mathcal{O}}_2}

\newcommand{\OTT}{{\mathcal{O}}_2 \otimes}

\newcommand{\Lem}[1]{Lemma~\ref{#1}}

\newcommand{\Zq}[1]{\Z_{#1}}
\newcommand{\Zqt}{\Zq{2}}

\title[Weak ideal property]{The weak ideal property and
 topological dimension zero}

\author{Cornel Pasnicu and N.~Christopher Phillips}

\date{6~April 2016}

\address{Department of Mathematics,
      The University of Texas at San Antonio,
      San Antonio TX 78249, USA.}

\email{Cornel.Pasnicu@utsa.edu}

\address{Department of Mathematics, University  of Oregon,
       Eugene OR 97403-1222, USA.}



\thanks{The work of the second author was supported by the
  US National Science Foundation under
  Grants DMS-1101742 and DMS-1501144.
  Some of the work was done during visits by both authors to
  the Institut Mittag-Leffler,
  and they are grateful to the Institut Mittag-Leffler
  for its hospitality and support.}

\begin{document}

\begin{abstract}
Following up on previous work,
we prove a number of results for C*-algebras
with the weak ideal property
or topological dimension zero,
and some results for C*-algebras with related properties.
Some of the more important results include:
\begin{itemize}
\item
The weak ideal property
implies topological dimension zero.
\item
For a separable C*-algebra~$A$,
topological dimension zero is equivalent to
${\operatorname{RR}} ({\mathcal{O}}_2 \otimes A) = 0$,
to $D \otimes A$ having the ideal property
for some (or any) Kirchberg algebra~$D$,
and to $A$ being residually hereditarily in
the class of all C*-algebras $B$ such that
${\mathcal{O}}_{\infty} \otimes B$
contains a nonzero projection.
\item
Extending the known result for ${\mathbb{Z}}_2$,
the classes of C*-algebras
with residual (SP),
which are residually hereditarily (properly) infinite,
or which are purely infinite and have the ideal property,
are closed under crossed products by arbitrary actions
of abelian $2$-groups.
\item
If $A$ and $B$ are separable,
one of them is exact,
$A$ has the ideal property,
and $B$ has the weak ideal property,
then $A \otimes_{\mathrm{min}} B$ has the weak ideal property.
\item
If $X$ is a totally disconnected locally compact Hausdorff space
and $A$ is a $C_0 (X)$-algebra
all of whose fibers have one of the weak ideal property,
topological dimension zero,
residual (SP),
or the combination of pure infiniteness and the ideal property,
then $A$ also has the corresponding property
(for topological dimension zero, provided $A$ is separable).
\item
Topological dimension zero,
the weak ideal property,
and the ideal property
are all equivalent
for a substantial class of separable C*-algebras including
all separable locally AH~algebras.
\item
The weak ideal property does not imply the ideal property
for separable $Z$-stable C*-algebras.
\end{itemize}
We give other related results,
as well as counterexamples to several other statements
one might hope for.
\end{abstract}

\maketitle

\indent
The weak ideal property
(recalled in Definition~\ref{D_5Y27_WIP} below)
was introduced in~\cite{PsnPh2};
it is the property for which there are good permanence results
(see Section~8 of~\cite{PsnPh2})
which seems to be closest to the ideal property.
(The ideal property fails to pass to extensions,
by Theorem~5.1 of~\cite{Psn1},
to corners, by Example~2.8 of~\cite{PsnPh1},
and to fixed point algebras under actions of~$\Zqt$,
by Example~2.7 of~\cite{PsnPh1}.
The weak ideal property does all of these.)
Topological dimension zero
was introduced in~\cite{BP09};
it is a non-Hausdorff version of total disconnectedness
of the primitive ideal space of a C*-algebra.
These two properties are related,
although not identical,
and the purpose of this paper is to study them
and their connections further.
Some of our results also involve the ideal property,
real rank zero,
and several forms of pure infiniteness.

For simple C*-algebras,
the appropriate regularity properties
(real rank zero, Property~(SP),
$Z$-stability, strict comparison,
pure infiniteness, etc.)\ %
are fairly well understood.
For nonsimple C*-algebras,
there are more such properties,
they are less well understood,
and it is not clear which of them are the ``right'' ones
to consider.
This paper is a contribution towards a better understanding
of some of these properties.

Even though it is not yet clear what the right
regularity properties in the nonsimple case are,
the properties we consider
(topological dimension zero, the weak ideal property,
and the ideal property) have at least proved to be valuable.
For example,
if $X$ is the primitive ideal space
of a separable C*-algebra, then $X$ has topological dimension zero
if and only if $X$ is the primitive ideal space of an AF~algebra.
(See Section~3 and the theorem in Section~5 of~\cite{BE}.)
If $A$ is a separable purely infinite C*-algebra, then $A$ has
real rank zero if and only if $A$ has topological dimension zero and is
$K_{0}$-liftable (Theorem~4.2 of~\cite{PR}).

Turning to the ideal property
(every ideal is generated, as an ideal, by its projections),
we consider AH algebras
(in the sense of~\cite{Ps2}:
the spaces used are connected finite complexes)
with the ideal property and with slow dimension growth.
Such algebras have stable rank one (Theorem 4.1 of~\cite{Ps2}),
and can be classified up to shape equivalence
by a K-theoretic invariant
(Theorem 2.15 of~\cite{Ps2}).
If $A$ is such an algebra and $K_{*} (A)$ is torsion free,
then $A$ is an AT~algebra, that is, a direct limit of finite
direct sums of matrix algebras over $C (S^{1})$
(Theorem~3.6 of~\cite{GJLP};
this paper uses a less restrictive definition of AH~algebra).
The stable rank one and AT~algebra results fail without
the ideal property.
(Counterexamples are easy, but too long for the introduction;
we present them at the beginning of Section~\ref{Sec_WIPIP}.)
Also, a
separable purely infinite C*-algebra has the ideal property if and only
if it has topological dimension zero (Proposition 2.11 of~\cite{PR}.)

The weak ideal property is much more recent.
As noted above,
it has better permanence properties than the ideal property.
Moreover,
under the hypotheses of the theorems above for the ideal property,
the weak ideal property actually implies the ideal property.
(See Theorem~\ref{T_5Z26_ClassForEq} and Theorem~\ref{T_5Y28_TDz}.)

We next describe our results.

We prove in Section~\ref{Sec_TDZ} that the weak ideal property
implies topological dimension zero in complete generality.
For separable C*-algebras
which are purely infinite in the sense of~\cite{KR},
it is equivalent to the ideal property
and to topological dimension zero.
A general separable C*-algebra~$A$ has topological dimension zero
\ifo{} $\OTT A$ has real rank zero;
this is also equivalent to $D \otimes A$ having the ideal property
for some (or any) Kirchberg algebra~$D$.
We rule out by example other results in this direction
that one might hope for.
Topological dimension zero, at least for separable C*-algebras,
is also equivalent to a property
of the sort considered in~\cite{PsnPh2}.
That is, there is an upwards directed class $\CC$
such that a separable C*-algebra~$A$ has topological dimension zero
\ifo{} $A$ is residually hereditarily in~$\CC$.
(See the end of the introduction for other examples of this
kind of property.)

In Section~\ref{Sec_PermCP},
we improve the closure properties under crossed products
of the class of C*-algebras
residually hereditarily in a class~$\CC$
by replacing an arbitrary action of $\Zqt$
with an arbitrary action of a finite abelian $2$-group.
(See Corollary~\ref{C_5X31_CCG2CP}.)
This refinement was overlooked in~\cite{PsnPh2}.
It applies to residual hereditary (proper) infiniteness
as well as to residual~(SP)
and to the combination of
pure infiniteness and the ideal property.
For the weak ideal property
and for topological dimension zero,
better results are already known
(Corollary 8.10 of~\cite{PsnPh2} and Theorem 3.17 of~\cite{PsnPh1}).
However,
for topological dimension zero,
in the separable case we remove the technical hypothesis
in Theorem 3.14 of~\cite{PsnPh1},
and show that if a finite group acts on a separable C*-algebra~$A$
and the fixed point algebra has topological dimension zero,
then $A$ has topological dimension zero.

Section~\ref{Sec_PermTP}
considers minimal tensor products.
For a tensor product to have the weak ideal property
or topological dimension zero,
it is usually necessary that both tensor factors
have the corresponding property.
In the separable case and with one factor exact,
this is sufficient for topological dimension zero.
We show by example that this result fails
without the exactness hypothesis.
For the weak ideal property,
we get only partial results:
if both factors are separable,
one is exact,
and one actually has the ideal property,
or if one factor is exact and one factor
has finite or Hausdorff primitive ideal space,
then the tensor product has the weak ideal property.

Proceeding to a $C_0 (X)$-algebra $A$,
we show that if $X$ is totally disconnected
and the fibers all have the weak ideal property,
topological dimension zero,
residual (SP),
or the combination of pure infiniteness and the ideal property,
then $A$ also has the corresponding property
(for topological dimension zero, provided $A$ is separable).
This result is the analog for these properties of
Theorem~2.1 of~\cite{Psn5}
(for real rank zero)
and Theorem~2.1 of~\cite{Psn6}
(for the ideal property),
but we do not assume that the $C_0 (X)$-algebra is continuous.
If $A$ is a separable continuous $C_0 (X)$-algebra
with nonzero fibers and $X$ is second countable,
then total disconnectedness of $X$ is also necessary.
This is in Section~\ref{Sec_PermBunTD}.
In the short Section~\ref{Sec_PermPI},
we consider locally trivial $C_0 (X)$-algebras
with fibers which are strongly purely infinite
in the sense of Definition~5.1 of~\cite{KR2},
and show (slightly generalizing the known result for $C_0 (X, B)$)
that $A$ is again strongly purely infinite.
In particular, this applies if the fibers are
separable, purely infinite,
and have topological dimension zero.

Section~\ref{Sec_WIPIP}
gives a substantial class of C*-algebras
for which the ideal property,
the weak ideal property,
and topological dimension zero
are all equivalent.
This class includes all separable locally AH~algebras
(using a somewhat restrictive definition:
the spaces involved must have only finitely many connected
components
and the maps must be injective).
It also includes further generalizations of separable AH~algebras,
such as separable LS algebras.
We also prove
that the weak ideal property implies the ideal property
for stable C*-algebras~$A$ such that $\Prim (A)$ is Hausdorff.
However, we show by example that
there is a $Z$-stable C*-algebra
with just one nontrivial ideal
which has the weak ideal property
but not the ideal property.

Ideals in C*-algebras
are assumed to be closed and two sided.
We write $\Zq{n}$ for $\Z / n \Z$,
since the $p$-adic integers will not appear.
If $\af \colon G \to \Aut (A)$
is a action of a group $G$ on a C*-algebra~$A$,
then $A^{\af}$ denotes the fixed point algebra.

Because of the role they play in this paper,
we recall the following definitions from~\cite{PsnPh2}.

\begin{dfn}[Definition 5.1 of~\cite{PsnPh2}]\label{D_5X30_UpClass}
Let $\CC$ be a class of C*-algebras.
We say that $\CC$ is
{\emph{upwards directed}} if
whenever $A$ is a C*-algebra which contains a subalgebra isomorphic to
an algebra in~$\CC$,
then $A \in \CC$.
\end{dfn}

\begin{dfn}[Definition 5.2 of~\cite{PsnPh2}]\label{D_5X30_ResHerC}
Let $\CC$ be an upwards directed class of C*-algebras,
and let $A$ be a C*-algebra.
\begin{enumerate}
\item\label{D_5X30_ResHerC_Her}
We say that $A$ is
{\emph{hereditarily in~$\CC$}}
if every nonzero \hsa{} of~$A$ is in~$\CC$.
\item\label{D_5X30_ResHerC_Res}
We say that $A$ is
{\emph{residually hereditarily in~$\CC$}}
if $A / I$ is hereditarily in~$\CC$ for every ideal $I \subset A$
with $I \neq A$.
\end{enumerate}
\end{dfn}

Section~5 of~\cite{PsnPh2}
gives permanence properties for a general condition defined this way.
We recall the conditions of this type
considered in~\cite{PsnPh2},
and add one more to be proved here.
\begin{enumerate}
\item\label{Item_5X31_PIIP}
Let $\CC$ be the class
of all C*-algebras which contain an infinite \pj.
Then $\CC$ is upwards directed (clear)
and a C*-algebra~$A$ is purely infinite and has the ideal property
\ifo{} $A$ is residually hereditarily in~$\CC$.
See the equivalence of conditions (ii) and~(iv)
of Proposition~2.11 of~\cite{PR}
(valid, as shown there, even when $A$ is not separable).
\item\label{Item_5X31_RPI}
Let $\CC$ be the class
of all C*-algebras which contain an infinite element.
Then $\CC$ is upwards directed (clear)
and a C*-algebra~$A$ is (residually) hereditarily infinite
(Definition~6.1 of~\cite{PsnPh2})
\ifo{} $A$ is residually hereditarily in~$\CC$.
(See Corollary~6.5 of~\cite{PsnPh2}.
We should point out that,
by Lemma 2.2(iii) of~\cite{KR},
if $D$ is a C*-algebra,
$B \subset D$ is a hereditary subalgebra,
and $a$ and $b$ are positive elements of~$B$
such that $a$ is Cuntz subequivalent to $b$ relative to~$D$,
then $a$ is Cuntz subequivalent to $b$ relative to~$B$.)
\item\label{Item_5X31_RHPrPI}
Let $\CC$ be the class
of all C*-algebras which contain a properly infinite element.
Then $\CC$ is upwards directed (clear)
and a C*-algebra~$A$ is (residually) hereditarily properly infinite
(Definition~6.2 of~\cite{PsnPh2})
\ifo{} $A$ is residually hereditarily in~$\CC$.
(Lemma 2.2(iii) of~\cite{KR} plays the same role here
as in~(\ref{Item_5X31_RPI}).)
\item\label{Item_5X31_RSP}
Let $\CC$ be the class
of all C*-algebras which
contain a nonzero \pj.
Then $\CC$ is upwards directed (clear).
A C*-algebra~$A$ has Property~(SP)
\ifo{} $A$ is hereditarily in~$\CC$,
and has residual~(SP)
(Definition~7.1 of~\cite{PsnPh2})
\ifo{} $A$ is residually hereditarily in~$\CC$.
(Both statements are clear.
Residual~(SP) appears,
without the name,
as a hypothesis in the discussion after Proposition~4.18 of~\cite{KR2}.)
\item\label{Item_5X31_WIP}
Let $\CC$ be the class
of all C*-algebras $B$ such that $K \otimes B$
contains a nonzero \pj.
Then $\CC$ is upwards directed (clear)
and a C*-algebra~$A$ has the weak ideal property
(Definition~8.1 of~\cite{PsnPh2};
recalled in Definition~\ref{D_5Y27_WIP} below)
\ifo{} $A$ is residually hereditarily in~$\CC$.
(This is shown at the beginning of the proof of
Theorem~8.5 of~\cite{PsnPh2}.)
\item\label{Item_5X31_TDimZ}
Let $\CC$ be the class
of all C*-algebras $B$ such that $\OT \otimes B$
contains a nonzero \pj.
Then $\CC$ is upwards directed.
(This is clear.)
A separable C*-algebra~$A$
has topological dimension zero
\ifo{} $A$ is residually hereditarily in~$\CC$.
(This will be proved in Theorem~\ref{T_5Y28_TDz} below.)
\end{enumerate}

\section{Topological dimension zero}\label{Sec_TDZ}

\indent
In this section,
we prove
(Theorem~\ref{T_5Y28_wipImpTDz})
that the weak ideal property implies topological dimension zero
for general C*-algebras.
We then give characterizations of topological dimension zero
for separable C*-algebras
(Theorem~\ref{T_5Y28_TDz})
and purely infinite separable C*-algebras (Theorem~\ref{T_5Y27_spiTDz}),
in terms of other properties of the algebra,
in terms of properties of their tensor
products with suitable Kirchberg algebras,
and (for general separable C*-algebras)
of the form of being residually hereditarily in
suitable upwards directed classes.
We also give two related counterexamples.
In particular,
there is a separable purely infinite unital nuclear C*-algebra~$A$
with one nontrivial ideal such that $\OTT A \cong A$
and $\RR (A) = 0$,
and an action $\af \colon \Zqt \to \Aut (A)$,
such that $\RR (C^* (\Zqt, A, \af)) \neq 0$.

We recall two definitions from~\cite{PsnPh1}.
We call a not necessarily Hausdorff space
{\emph{locally compact}}
if the compact (but not necessarily closed) neighborhoods
of every point $x \in X$ form a neighborhood base at~$x$.

\begin{dfn}[Remark 2.5(vi) of~\cite{BP09};
 Definition~3.2 of~\cite{PsnPh1}]\label{D_5Y27_TDZero311}
Let $X$ be a locally compact
but not necessarily Hausdorff topological space.
We say that $X$ has
{\emph{topological dimension zero}}
if for every $x \in X$ and every open set $U \subset X$
such that $x \in U$,
there exists a compact open (but not necessarily closed)
subset $Y \subset X$
such that $x \in Y \subset U$.
(Equivalently,
$X$ has a base for its topology consisting of
subsets which are compact and open, but not necessarily closed.)
We further say that a C*-algebra~$A$
has
{\emph{topological dimension zero}}
if $\Prim (A)$ has topological dimension zero.
\end{dfn}

\begin{dfn}[Definition~3.4 of~\cite{PsnPh1}]\label{D_5Y27_COExh311}
Let $X$ be a not necessarily Hausdorff topological space.
A {\emph{compact open exhaustion}}
of~$X$
is an increasing net
$(Y_{\ld})_{\ld \in \Ld}$
of compact open subsets $Y_{\ld} \subset X$
such that $X = \bigcup_{\ld \in \Ld} Y_{\ld}$.
\end{dfn}

We further recall (Lemma 3.10 of~\cite{PsnPh1};
see Definition~3.9 of~\cite{PsnPh1}
or page~53 of~\cite{PR} for the original definition)
that if $A$ is a C*-algebra and $I \subset A$ is an ideal,
then $I$ is compact \ifo{} $\Prim (I)$ is a compact open
(but not necessarily closed) subset of $\Prim (A)$.

Finally, we recall the definition of the weak ideal property.

\begin{dfn}[Definition~8.1 of~\cite{PsnPh2}]\label{D_5Y27_WIP}
Let $A$ be a C*-algebra.
We say that $A$ has the {\emph{weak ideal property}}
if whenever $I \subset J \subset K \otimes A$
are ideals in $K \otimes A$
such that $I \neq J$,
then $J / I$
contains a nonzero \pj.
\end{dfn}

\begin{lem}\label{L_5Y27_OnePj}
Let A be a C*-algebra with the weak ideal property.
Let $I_1 \subset I_2 \subset A$
be ideals with $I_1 \neq I_2$.
Then there exists an ideal $J \subset A$
with $I_1 \subsetneqq J$
and such that $K \otimes (J / I_1)$
is generated as an
ideal by a single nonzero projection.
\end{lem}

\begin{proof}
Since $A$ has the weak ideal property and
$K \otimes (I_{2}/I_{1}) \neq 0$,
there is a \nzp{} $e \in K \otimes (I_{2}/I_{1})$.
Let $I \subset  K \otimes (I_{2}/I_{1})$
be the ideal generated by~$e$.
Then there is an ideal $J \subset A$ with
$I_1 \subset J \subset I_2$ such that $I = K \otimes (J / I_1)$.
Since $J / I_1 \neq 0$, it follows that $J \neq I_1$.
\end{proof}

\begin{lem}\label{L_5Y27_GetCptOpen}
Let $A$ be a C*-algebra,
let $F \subset A$ be a finite set of \pj{s},
and let $I \subset A$ be the ideal generated by~$F$.
Then $\Prim (I)$ is a compact open subset of $\Prim (A)$.
\end{lem}

\begin{proof}
This can be shown by using the same argument as in (iii) implies~(i)
in the proof of Proposition~2.7 of~\cite{PR}.
However,
we can give a more direct proof
(not involving the Pedersen ideal).
As there, we prove that $I$ is compact
(as recalled after Definition~\ref{D_5Y27_COExh311}).
So let $(I_{\ld})_{\ld \in \Ld}$
be an increasing net of ideals in~$A$
such that ${\overline{\bigcup_{\ld \in \Ld} I_{\ld} }} = I$.
Standard functional calculus arguments produce
$\ep > 0$ such that whenever $B$ is a C*-algebra,
$C \subset B$ is a subalgebra,
and $p \in B$ is a \pj{} such that $\dist (p, C) < \ep$,
then there is a \pj~$q \in C$
such that $\| q - p \| < 1$,
and in particular $q$ is \mvnt~$q$.
Write $F = \{ p_1, p_2, \ldots, p_n \}$.
Choose $\ld \in \Ld$
such that $\dist (p_j, I_{\ld}) < \ep$
for $j = 1, 2, \ldots, n$.
Let $q_1, q_2, \ldots, q_n \in I_{\ld}$
be projections obtained from the choice of~$\ep$.
Then there are partial isometries $s_1, s_2, \ldots, s_n \in A$
such that $p_j = s_j q_j s_j^*$ for $j = 1, 2, \ldots, n$.
So $p_1, p_2, \ldots, p_n \in I_{\ld}$,
whence $I_{\ld} = I$.
This completes the proof.
\end{proof}

\begin{lem}\label{L_5Y27_CptExh}
Let $A$ be a C*-algebra,
and let $I \subset A$ be an ideal.
Suppose that there is a collection $(I_{\ld})_{\ld \in \Ld}$
(not necessarily a net)
of ideals in~$A$
such that $I$ is the ideal generated by $\bigcup_{\ld \in \Ld} I_{\ld}$
and such that $\Prim (I_{\ld})$ has a compact open exhaustion
(as in Definition~\ref{D_5Y27_COExh311})
for every $\ld \in \Ld$.
Then $\Prim (I)$ has a compact open exhaustion.
\end{lem}

\begin{proof}
It is easily checked that a union of open sets
with compact open exhaustions
also has a compact open exhaustion.
\end{proof}

\begin{prp}\label{P_5Y28_LargestCOE}
Let $A$ be a C*-algebra.
Then there is a largest ideal $I \subset A$
such that $\Prim (I)$ has a compact open exhaustion.
\end{prp}

\begin{proof}
Let $I$ be the closure
of the union of all ideals $J \subset A$
such that $\Prim (J)$ has a compact open exhaustion.
Then $\Prim (I)$ has a compact open exhaustion
by Lemma~\ref{L_5Y27_CptExh}.
\end{proof}

\begin{thm}\label{T_5Y28_wipImpTDz}
Let $A$ be a C*-algebra with the weak ideal property.
Then $A$ has topological dimension zero.
\end{thm}

\begin{proof}
We will show that for every ideal $I \subset A$,
the subset $\Prim (I)$ has a compact open exhaustion.
The desired conclusion will then follow
from Lemma~3.6 of~\cite{PsnPh1}.

So let $I \subset A$ be an ideal.
By Proposition~\ref{P_5Y28_LargestCOE},
there is a largest ideal $J \subset I$ such that
$\Prim (J)$ has a compact open exhaustion.
We prove that $J = I$.
Suppose not.
Use Lemma~\ref{L_5Y27_OnePj}
to find an ideal $N \subset I$
with $J \subsetneqq N$
and such that $K \otimes (N / J)$ is generated by one \nzp.
Then $\Prim (K \otimes (N / J))$ is a compact open subset
of $\Prim (K \otimes (I / J))$
by Lemma~\ref{L_5Y27_GetCptOpen}.
So $\Prim (N / J)$ is a compact open subset
of $\Prim (I / J)$.
Since $\Prim (J)$ has a compact open exhaustion,
we can apply Lemma 3.7 of~\cite{PsnPh1}
(taking $U = \Prim (J)$)
to deduce that $\Prim (N)$ has a compact open exhaustion.
Since $J \subsetneqq N$,
we have a contradiction.
Thus $J = I$,
and $\Prim (I)$ has a compact open exhaustion.
\end{proof}

The list of equivalent conditions in the next theorem
extends the list
in Corollary 4.3 of~\cite{PR},
by adding condition~(\ref{T_5Y27_spiTDz_wip}).
As discussed in the introduction,
this condition is better behaved
than the related condition~(\ref{T_5Y27_spiTDz_IP}).

\begin{thm}\label{T_5Y27_spiTDz}
Let $A$ be a separable C*-algebra
which is purely infinite in the sense of Definition 4.1 of~\cite{KR}.
Then \tfae:
\begin{enumerate}
\item\label{T_5Y27_spiTDz_rrz}
${\mathcal{O}}_2 \otimes A$ has real rank zero.
\item\label{T_5Y27_spiTDz_O2IP}
${\mathcal{O}}_2 \otimes A$ has the ideal property.
\item\label{T_5Y27_spiTDz_tdZ}
$A$ has topological dimension zero.
\item\label{T_5Y27_spiTDz_IP}
$A$ has the ideal property.
\item\label{T_5Y27_spiTDz_wip}
$A$ has the weak ideal property.
\end{enumerate}
\end{thm}

\begin{proof}
The equivalence of conditions
(\ref{T_5Y27_spiTDz_rrz}),
(\ref{T_5Y27_spiTDz_O2IP}),
(\ref{T_5Y27_spiTDz_tdZ}),
and (\ref{T_5Y27_spiTDz_IP})
is Corollary~4.3 of~\cite{PR}.
That (\ref{T_5Y27_spiTDz_IP})
implies~(\ref{T_5Y27_spiTDz_wip})
is trivial.
That (\ref{T_5Y27_spiTDz_wip}) implies~(\ref{T_5Y27_spiTDz_tdZ})
is Theorem~\ref{T_5Y28_wipImpTDz}.
\end{proof}

We presume that Theorem~\ref{T_5Y27_spiTDz}
holds without separability.
However,
some of the results used in the proof
of Corollary~4.3 of~\cite{PR}
are only known in the separable case,
and it seems likely to require some work to generalize them.

Recall that a Kirchberg algebra
is a simple separable nuclear purely infinite C*-algebra.

\begin{thm}\label{T_5Y28_TDz}
Let $A$ be a separable C*-algebra.
Then \tfae:
\begin{enumerate}
\item\label{T_5Y28_TDz_tdZ}
$A$ has topological dimension zero.
\item\label{T_5Y28_TDz_O2rrz}
${\mathcal{O}}_2 \otimes A$ has real rank zero.
\item\label{T_5Y28_TDz_O2IP}
${\mathcal{O}}_2 \otimes A$ has the ideal property.
\item\label{T_5Y28_TDz_O2wip}
${\mathcal{O}}_2 \otimes A$ has the weak ideal property.
\item\label{T_5Y28_TDz_OInfIP}
${\mathcal{O}}_{\infty} \otimes A$ has the ideal property.
\item\label{T_5Y28_TDz_OInftywip}
${\mathcal{O}}_{\infty} \otimes A$ has the weak ideal property.
\item\label{T_5Y28_TDz_SomeKwip}
There exists a Kirchberg algebra $D$
such that $D \otimes A$ has the weak ideal property.
\item\label{T_5Y28_TDz_AllKip}
For every Kirchberg algebra $D$,
the algebra $D \otimes A$ has the ideal property.
\item\label{T_5Y28_TDz_RHO2}
$A$ is residually hereditarily in the class
of all C*-algebras $B$ such that $\OT \otimes B$
contains a nonzero \pj.
\item\label{T_5Y28_TDz_RHKtO2}
$A$ is residually hereditarily in the class
of all C*-algebras $B$ such that $K \otimes \OT \otimes B$
contains a nonzero \pj.
\item\label{T_5Y28_TDz_RHOInfty}
$A$ is residually hereditarily in the class
of all C*-algebras $B$ such that ${\mathcal{O}}_{\infty} \otimes B$
contains a nonzero \pj.
\end{enumerate}
\end{thm}

We presume that Theorem~\ref{T_5Y28_TDz}
also holds without separability.

To put conditions
(\ref{T_5Y28_TDz_RHO2}),
(\ref{T_5Y28_TDz_RHKtO2}),
and~(\ref{T_5Y28_TDz_RHOInfty}) in context,
we point out that it is clear that
the classes used in them are upwards directed
in the sense of Definition~\ref{D_5X30_UpClass}.
Applying the results of Section~5 of~\cite{PsnPh2}
does not give any closure properties
for the collection of C*-algebras
with topological dimension zero
which are not already known.
We do get something new,
which at least implicitly is related this characterization;
see Theorem~\ref{T_5Z20_FPAlg} below.

The conditions in Theorem~\ref{T_5Y28_TDz}
are not equivalent to $A$ having the weak ideal property,
since there are nonzero simple separable C*-algebras~$A$,
such as those classified in~\cite{Rzk},
for which $K \otimes A$ has no nonzero projections.
They are also not equivalent to
$\RR ({\mathcal{O}}_{\infty} \otimes A) = 0$.
See Example~\ref{E_5Y28_OINotRRZ} below.

\begin{proof}[Proof of Theorem~\ref{T_5Y28_TDz}]
Since $A$ has topological dimension zero \ifo{}
${\mathcal{O}}_2 \otimes A$ has topological dimension zero,
and since ${\mathcal{O}}_2 \otimes A$ is purely infinite
(by Proposition~4.5 of~\cite{KR}),
the equivalence of (\ref{T_5Y28_TDz_tdZ}),
(\ref{T_5Y28_TDz_O2rrz}),
(\ref{T_5Y28_TDz_O2IP}),
and (\ref{T_5Y28_TDz_O2wip})
follows by applying Theorem~\ref{T_5Y27_spiTDz}
to ${\mathcal{O}}_2 \otimes A$.
Since ${\mathcal{O}}_{\infty} \otimes A$ is purely infinite
(by Proposition~4.5 of~\cite{KR})
and $\OTT {\mathcal{O}}_{\infty} \cong {\mathcal{O}}_2$,
the equivalence of (\ref{T_5Y28_TDz_O2IP}),
(\ref{T_5Y28_TDz_OInfIP}),
and (\ref{T_5Y28_TDz_OInftywip})
follows by applying Theorem~\ref{T_5Y27_spiTDz}
to ${\mathcal{O}}_{\infty} \otimes A$.

We prove the equivalence of (\ref{T_5Y28_TDz_tdZ})
and (\ref{T_5Y28_TDz_RHO2}).
Let $\CC$ be the class
of all C*-algebras $B$ such that $\OT \otimes B$
contains a nonzero \pj.

Assume that $A$ has topological dimension zero;
we prove that $A$ is residually hereditarily in~$\CC$.
Let $I \subset A$ be an ideal,
and let $B \subset A / I$ be a nonzero hereditary subalgebra.
Then $A / I$ has topological dimension zero by Proposition~2.6
of~\cite{BP09} and Lemma~3.6 of~\cite{PsnPh1}.
It follows from Lemma 3.3 of~\cite{PsnPh1}
that $B$ has topological dimension zero.
Use (\ref{T_5Y27_spiTDz_tdZ}) implies~(\ref{T_5Y27_spiTDz_rrz})
in Theorem~\ref{T_5Y27_spiTDz}
to conclude that $\OTT B$ contains a \nzp.

Conversely,
assume that $A$ is residually hereditarily in~$\CC$.
We actually prove that $\OTT A$ has the weak ideal property.
By (\ref{T_5Y27_spiTDz_wip}) implies~(\ref{T_5Y27_spiTDz_tdZ})
in Theorem~\ref{T_5Y27_spiTDz},
and since $\OTT A$ is purely infinite
(by Proposition~4.5 of~\cite{KR}),
it will follow that $\OTT A$ has topological dimension zero.
Since $\Prim (\OTT A) \cong \Prim (A)$,
it will follow that $A$ has topological dimension zero.

Thus,
let $I \subset J \subset \OTT A$
be ideals such that $J \neq I$;
we must show that $K \otimes (J / I)$
contains a \nzp.
Since $\OT$ is simple and nuclear,
there are ideals $I_0 \subset J_0 \subset A$
such that $I = \OTT I_0$ and $J = \OTT J_0$;
moreover,
$J / I \cong \OTT (J_0 / I_0)$.
Since $J_0 / I_0$ is a nonzero hereditary subalgebra of $A / I_0$,
the definition of being hereditarily in~$\CC$
implies that $\OTT (J_0 / I_0)$
contains a \nzp,
so $K \otimes (J / I) \cong K \otimes \OTT (J_0 / I_0)$ does also.
This completes the proof of the equivalence of (\ref{T_5Y28_TDz_tdZ})
and (\ref{T_5Y28_TDz_RHO2}).

We prove equivalence of (\ref{T_5Y28_TDz_RHO2})
and (\ref{T_5Y28_TDz_RHOInfty})
by showing that the two classes involved are equal,
that is,
by showing that if $B$ is any C*-algebra,
then ${\mathcal{O}}_2 \otimes B$
contains a \nzp{}
\ifo{} ${\mathcal{O}}_{\infty} \otimes B$
contains a \nzp.
If ${\mathcal{O}}_2 \otimes B$ contains a \nzp,
use an injective (nonunital) \hm{}
${\mathcal{O}}_2 \to {\mathcal{O}}_{\infty}$
to produce an injective \hm{} of the minimal tensor products
${\mathcal{O}}_2 \otimes_{\mathrm{min}} B
  \to {\mathcal{O}}_{\infty} \otimes_{\mathrm{min}} B$.
Since ${\mathcal{O}}_2$ and ${\mathcal{O}}_{\infty}$ are nuclear,
we have an injective \hm{}
${\mathcal{O}}_2 \otimes B
  \to {\mathcal{O}}_{\infty} \otimes B$,
and hence a \nzp{} in ${\mathcal{O}}_{\infty} \otimes B$.
Using an injective (unital) \hm{}
from ${\mathcal{O}}_{\infty}$ to ${\mathcal{O}}_2$,
the same argument also shows that
if ${\mathcal{O}}_{\infty} \otimes B$
contains a \nzp{}
then so does ${\mathcal{O}}_{2} \otimes B$.

The proof of the equivalence of (\ref{T_5Y28_TDz_RHO2})
and (\ref{T_5Y28_TDz_RHKtO2})
is essentially the same
as in the previous paragraph,
using injective \hm{s}
\[
{\mathcal{O}}_2 \longrightarrow K \otimes {\mathcal{O}}_2
\andeqn
K \otimes {\mathcal{O}}_2
 \longrightarrow {\mathcal{O}}_2 \otimes {\mathcal{O}}_2
 \stackrel{\cong}\longrightarrow {\mathcal{O}}_2.
\]

We have now proved the equivalence of all the conditions
except (\ref{T_5Y28_TDz_SomeKwip}) and~(\ref{T_5Y28_TDz_AllKip}).

It is trivial
that (\ref{T_5Y28_TDz_OInftywip}) implies (\ref{T_5Y28_TDz_SomeKwip})
and
that (\ref{T_5Y28_TDz_AllKip}) implies (\ref{T_5Y28_TDz_OInfIP}).

Assume (\ref{T_5Y28_TDz_SomeKwip}),
so that there is a Kirchberg algebra $D_0$
such that $D_0 \otimes A$ has the weak ideal property.
We prove~(\ref{T_5Y28_TDz_AllKip}).
Let $D$ be any Kirchberg algebra.
By Theorem~\ref{T_5Y28_wipImpTDz},
the algebra $D_0 \otimes A$ has topological dimension zero.
Since
\[
\Prim (D_0 \otimes A)
 \cong \Prim (A)
 \cong \Prim (D \otimes A),
\]
$D \otimes A$ has topological dimension zero.
Apply the already proved implication from
(\ref{T_5Y28_TDz_tdZ})
to~(\ref{T_5Y28_TDz_OInfIP}) with $D \otimes A$ in place of~$A$,
concluding that ${\mathcal{O}}_{\infty} \otimes D \otimes A$
has the ideal property.
Since ${\mathcal{O}}_{\infty} \otimes D \cong D$
(Theorem 3.15 of~\cite{KP1}),
we see that $D \otimes A$
has the ideal property.
\end{proof}

A naive look at condition~(\ref{T_5Y27_spiTDz_rrz})
of Theorem~\ref{T_5Y27_spiTDz}
and the permanence
properties for C*-algebras which are residually hereditarily in some
class~$\CC$
(see Corollary~5.6 and Theorem~5.3 of~\cite{PsnPh2})
might suggest that if ${\mathcal{O}}_{\infty} \otimes A$
has real rank zero and one has an arbitrary action of $\Zqt$
on ${\mathcal{O}}_{\infty} \otimes A$,
or a spectrally free
(Definition~1.3 of~\cite{PsnPh2}) action of any
discrete group
on ${\mathcal{O}}_{\infty} \otimes A$,
then the crossed product should also have real rank zero.
This is false.
We give an example of
a nonsimple purely infinite unital nuclear C*-algebra~$A$
satisfying the Universal Coefficient Theorem
(in fact, with $\OTT A \cong A$),
with exactly one nontrivial ideal,
and such that $\RR (A) = 0$,
and an action $\af \colon \Zqt \to \Aut (A)$,
such that $C^* (\Zqt, A, \af)$
does not have real rank zero.

To put our example in context,
we recall the following.
First, Example~9 of~\cite{Ell3}
gives an example of a pointwise outer action~$\af$
of $\Zqt$
on a simple unital AF~algebra~$A$
such that $C^* (\Zqt, A, \af)$ does not have real rank zero.
Second,
by Corollary~4.4 of~\cite{JO},
if $A$ is purely infinite and simple,
then for any action $\af \colon \Zqt \to \Aut (A)$
the crossed product is again purely infinite.
If $\af$ is pointwise outer,
then $C^* (\Zqt, A, \af)$ is again simple,
so automatically has real rank zero.
Otherwise, $\af$ must be an inner action.
(See Lemma~\ref{L_5Z26_Inner} below.)
Then $C^* (\Zqt, A, \af) \cong A \oplus A$,
so has real rank zero.
Thus, no such example is possible
when $A$ is purely infinite and simple.
Third,
it is possible for $A$ to satisfy $\OTT A \cong A$
but to have
$\OTT C^* (\Zqt, A, \af) \not\cong C^* (\Zqt, A, \af)$.
See Lemma~4.7 of~\cite{Iz1},
where this happens with $A = {\mathcal{O}}_2$.

The following lemma is well known,
but we don't know a reference.

\begin{lem}\label{L_5Z26_Inner}
Let $A$ be a simple C*-algebra,
let $G$ be a finite cyclic group,
and let $\af \colon G \to \Aut (A)$
be an action of $G$ on~$A$.
Let $g_0 \in G$ be a generator of~$G$.
If $\af_{g_0}$ is inner,
then $\af$ is an inner action,
that is,
there is a \hm{} $g \mapsto u_g$
from $G$ to the unitary group of $M (A)$
such that $\af_g (a) = u_g a u_g^*$
for all $g \in G$ and $a \in A$.
\end{lem}

\begin{proof}
Let $n$ be the order of~$G$.

By hypothesis,
there is a unitary $v \in M (A)$
such that $\af_{g_0} (a) = v a v^*$ for all $a \in A$.
Then
$a = \af_{g_0}^n (a) = v^n a v^{-n}$
for all $a \in A$.
Simplicity of $A$ implies that the center of $M (A)$
contains only scalars,
so there is $\ld \in S^1$
such that $v^n = \ld \cdot 1$.
Now choose $\om \in S^1$ such that $\om^{n} = \ld^{-1}$,
giving $(\om v)^n = 1$.
Define $u_{g_0^k} = \om^k v^k$ for
$k = 0, 1, \ldots, n - 1$.
\end{proof}

\begin{exa}\label{E_5Y01_CPNotRRZ}
There is a separable purely infinite unital nuclear C*-algebra~$A$
with exactly one nontrivial ideal~$I$,
satisfying the Universal Coefficient Theorem,
satisfying $\OTT A \cong A$,
and such that $\RR (A) = 0$,
and there is an action $\af \colon \Zqt \to \Aut (A)$
such that $\RR (C^* (\Zqt, A, \af)) \neq 0$.
Moreover, $\af$ is strongly pointwise outer
in the sense of Definition~4.11 of~\cite{Phfgs}
(Definition~1.1 of~\cite{PsnPh2})
and spectrally free in the sense of Definition~1.3 of~\cite{PsnPh2}.

To start the construction,
let $\nu \colon \Zqt \to \Aut ( \OT )$
be the action considered in Lemma~4.7 of~\cite{Iz1}.
Define $B = C^* (\Zqt, \OT, \nu)$.
Lemma~4.7 of~\cite{Iz1}
implies that $B$ is a Kirchberg algebra
(simple, separable, nuclear, and purely infinite)
which is unital and satisfies the Universal Coefficient Theorem,
and moreover that $K_0 (B) \cong \Z \big[ \frac{1}{2} \big]$
and $K_1 (B) = 0$.

Let $P$ be the unital Kirchberg algebra
satisfying the Universal Coefficient Theorem,
$K_0 (P) = 0$, and $K_1 (P) \cong \Z$.
The K\"{u}nneth formula (Theorem~4.1 of~\cite{Sc2})
implies that $K_0 ( P \otimes {\mathcal{O}}_4 ) = 0$
and $K_1 ( P \otimes {\mathcal{O}}_4 ) \cong \Zq{3}$.

Since ${\mathcal{O}}_4$
satisfies the Universal Coefficient Theorem,
and since $K \otimes P \otimes {\mathcal{O}}_4$ and ${\mathcal{O}}_4$
are separable and nuclear,
every possible six term exact sequence
\[
\begin{CD}
K_0 (P \otimes {\mathcal{O}}_4) @>>> M_0
                @>>> K_0 ({\mathcal{O}}_4)   \\
@A{\exp}AA & &  @VV{\partial}V            \\
K_1 ({\mathcal{O}}_4) @<<< M_1
       @<<< K_1 (P \otimes {\mathcal{O}}_4)
\end{CD}
\]
(for any possible choice of abelian groups $M_0$ and $M_1$
and \hm{s} $\exp$ and~$\partial$)
is realized as the K-theory of an exact sequence
\begin{equation}\label{Eq_5Y01_Sq}
0 \longrightarrow K \otimes P \otimes {\mathcal{O}}_4
  \longrightarrow D
  \longrightarrow {\mathcal{O}}_4
  \longrightarrow 0,
\end{equation}
in which $D$ is unital,
$K_0 (D) \cong M_0$, and $K_1 (D) \cong M_1$.
Choose the exact sequence~(\ref{Eq_5Y01_Sq})
such that the connecting map
\begin{equation}\label{Eq_4Y01_Conn}
\exp \colon K_0 ( {\mathcal{O}}_4 )
 \to K_1 \big( P \otimes {\mathcal{O}}_4 \big)
\end{equation}
is an isomorphism.
Define $A = \OTT D$.
Let $\io \colon \Zqt \to \Aut (D)$
be the trivial action,
and let $\af = \nu \otimes \io \colon \Zqt \to \Aut (A)$
be the obvious action on the tensor product.

It follows from
the isomorphisms $\OTT {\mathcal{O}}_4 \cong \OT$
and $\OTT P \otimes {\mathcal{O}}_4 \cong \OT$
that $A$ fits into an exact sequence
\[
0 \longrightarrow K \otimes \OT
  \longrightarrow A
  \longrightarrow {\mathcal{O}}_2
  \longrightarrow 0.
\]
Theorem 3.14 and Corollary 3.16 of~\cite{BP}
therefore imply that $\RR (A) = 0$.
There is also an exact sequence of crossed products
\[
0 \longrightarrow
      C^* \big( \Zqt, \, {\mathcal{O}}_2 \otimes
              K \otimes P \otimes {\mathcal{O}}_4 \big)
  \longrightarrow C^* (\Zqt, A, \af)
  \longrightarrow
    C^* \big( \Zqt, \, {\mathcal{O}}_2 \otimes {\mathcal{O}}_4 \big)
  \longrightarrow 0,
\]
in which the actions on the ideal and quotient are the tensor product
of~$\nu$ and the trivial action.
This sequence reduces to
\[
0 \longrightarrow B \otimes K \otimes P \otimes {\mathcal{O}}_4
  \longrightarrow B \otimes D
  \longrightarrow B \otimes {\mathcal{O}}_4
  \longrightarrow 0,
\]
in which the maps are gotten from those of~(\ref{Eq_5Y01_Sq})
by tensoring them with $\id_B$.
It follows from K\"{u}nneth formula (Theorem~4.1 of~\cite{Sc2})
that
\[
K_0 (B \otimes {\mathcal{O}}_4)
 \cong K_1 \big( B \otimes K \otimes P \otimes {\mathcal{O}}_4 \big)
 \cong \Z \big[ \tfrac{1}{2} \big] \otimes \Zq{3}
 \cong \Zq{3}.
\]
By naturality,
the connecting map
\[
K_0 ( B \otimes {\mathcal{O}}_4 )
 \to K_1 \big( B \otimes K \otimes P \otimes {\mathcal{O}}_4 \big)
\]
is the tensor product of the
isomorphism~(\ref{Eq_4Y01_Conn})
with $\id_{\Z [ \frac{1}{2} ]}$,
and is hence nonzero.
Since every class in $K_0 ( B \otimes {\mathcal{O}}_4 )$
is represented by a \pj{}
in $B \otimes {\mathcal{O}}_4$,
it follows from the six term exact sequence in K-theory
that projections in $B \otimes {\mathcal{O}}_4$
need not lift to \pj{s} in $B \otimes D$.
Theorem 3.14 of~\cite{BP}
therefore implies that $\RR (B \otimes D) \neq 0$.
Thus $\RR (C^* (\Zqt, A, \af)) \neq 0$.

It remains to prove strong pointwise outerness and spectral freeness.
In our case,
these are equivalent by Theorem~1.16 of~\cite{PsnPh2},
so we prove strong pointwise outerness.
This reduces to proving that automorphisms
of ${\mathcal{O}}_2 \otimes K \otimes P \otimes {\mathcal{O}}_4$
and
${\mathcal{O}}_2 \otimes {\mathcal{O}}_4$
coming from the nontrivial element of $\Zqt$ are outer.
The automorphism of ${\mathcal{O}}_2$
coming from the action $\nu$
and the nontrivial element of $\Zqt$ is outer,
since otherwise the action would be inner by Lemma~\ref{L_5Z26_Inner},
so crossed product would be
${\mathcal{O}}_2 \oplus {\mathcal{O}}_2$.
We can now apply Proposition 1.19 of~\cite{PsnPh2} twice,
both times using $\nu \colon \Zqt \to \Aut ( {\mathcal{O}}_2 )$
in place of $\af \colon G \to \Aut (A)$,
and in one case using $K \otimes P \otimes {\mathcal{O}}_4$
in place of~$B$ and in the other case using~${\mathcal{O}}_4$.
\end{exa}

We would like to get outerness from
Theorem~1 of~\cite{Was},
but that theorem is only stated for unital C*-algebras.

\begin{exa}\label{E_5Y28_OINotRRZ}
There is a separable purely infinite unital nuclear C*-algebra~$A$
with exactly one nontrivial ideal
and which has the ideal property
but such that ${\mathcal{O}}_{\infty} \otimes A$
does not have real rank zero.

Let $D$ be as in~(\ref{Eq_5Y01_Sq})
in Example~\ref{E_5Y01_CPNotRRZ},
with the property that the connecting map
in~(\ref{Eq_4Y01_Conn}) is nonzero.
Set $A = {\mathcal{O}}_{\infty} \otimes D$.
Since
${\mathcal{O}}_{\infty} \otimes K \otimes P \otimes {\mathcal{O}}_4$
and ${\mathcal{O}}_{\infty} \otimes{\mathcal{O}}_4$
have the weak ideal property (for trivial reasons),
it follows from Theorem 8.5(5) of~\cite{PsnPh2}
that $A$ has the weak ideal property,
and from Theorem~\ref{T_5Y27_spiTDz}
that $A$ has the ideal property.
However,
$A$ is by construction not $K_0$-liftable
in the sense of Definition~3.1 of~\cite{PR},
so Corollary~4.3 of~\cite{PR} implies that
${\mathcal{O}}_{\infty} \otimes A$
(which is of course isomorphic to~$A$)
does not have real rank zero.
\end{exa}

\section{Permanence properties for crossed products}\label{Sec_PermCP}

\indent
In~\cite{PsnPh2} we proved that if
$\CC$ is an upwards directed class of C*-algebras,
$\af$ is a completely
arbitrary action of $\Zqt$ on a C*-algebra~$A$,
and $A^{\af}$ is (residually) hereditarily in~$\CC$,
then $A$ is (residually) hereditarily in~$\CC$.
(See Theorem~5.5 of~\cite{PsnPh2}.)
In particular,
by considering dual actions,
it follows (Corollary~5.6 of~\cite{PsnPh2})
that
crossed products by arbitrary actions of $\Zqt$
preserve the class of C*-algebras
which are (residually) hereditarily in~$\CC$.
Here,
we show how one can easily extend the first result
to arbitrary groups of order a power of~$2$
and the second result
to arbitrary abelian groups of order a power of~$2$.
This should have been done in~\cite{PsnPh2},
but was overlooked there.
We believe these results should be true for any finite group
in place of~$\Zqt$,
or at least any finite abelian group,
but we don't know how to prove them in this generality.

The following lemma is surely well known.

\begin{lem}\label{L_5X31_ActQuot}
Let $G$ be a topological group,
let $A$ be a C*-algebra,
and let $\af \colon G \to \Aut (A)$
be an action of $G$ on~$A$.
Let $N \subset G$ be a closed normal subgroup.
Then there is an action
${\ov{\af}} \colon G / N \to \Aut (A^{\af |_N} )$
such that for $g \in G$ and $a \in A^{\af |_N}$
we have
${\ov{\af}}_{g N} (a) = \af_g (a)$.
Moreover,
$(A^{\af |_N})^{\ov{\af}} = A^{\af}$.
\end{lem}

\begin{proof}
The only thing requiring proof is that
if $g \in G$ and $a \in A^{\af |_N}$
then $\af_g (a) \in A^{\af |_N}$.
So let $k \in N$.
Since $g^{-1} k g \in N$,
we get
\[
\af_k ( \af_g (a))
 = \af_g \big( \af_{g^{-1} k g} (a) \big)
 = \af_g (a).
\]
This completes the proof.
\end{proof}

\begin{thm}\label{T_5X30_CCG}
Let $\CC$ be an upwards directed class of C*-algebras.
Let $G$ be a finite $2$-group,
and
let $\af \colon G \to \Aut (A)$
be an arbitrary action of $G$ on a C*-algebra~$A$.
\begin{enumerate}
\item\label{T_5X30_CCG_Her}
If $A^{\af}$ is hereditarily in~$\CC$,
then $A$ is hereditarily in~$\CC$.
\item\label{T_5X30_CCG_Res}
If $A^{\af}$ is residually hereditarily in~$\CC$,
then $A$ is residually hereditarily in~$\CC$.
\end{enumerate}
\end{thm}

\begin{proof}
We prove both parts at once.

We use induction on the number $n \in \Nz$
such that the order of~$G$ is~$2^n$.
When $n = 0$,
the statement is trivial.
So assume $n \in \Nz$,
the statement is known for all groups of order~$2^n$,
$G$ is a group with $\card (G) = 2^{n + 1}$,
$A$ is a C*-algebra,
$\af \colon G \to \Aut (A)$ is an action,
and $A^{\af}$ is (residually) hereditarily in~$\CC$.
The Sylow Theorems provide
a subgroup $N \subset G$
such that $\card (N) = 2^n$.
Since $N$ has index~$2$,
$N$ must be normal.
Let ${\ov{\af}} \colon G / N \to \Aut (A^{\af |_N} )$
be as in \Lem{L_5X31_ActQuot}.
Then $(A^{\af |_N})^{\ov{\af}} = A^{\af}$
is (residually) hereditarily in~$\CC$.
Since $G / N \cong \Zqt$,
it follows from Theorem~5.5 of~\cite{PsnPh2}
that $A^{\af |_N}$
is (residually) hereditarily in~$\CC$.
The induction hypothesis
now implies that $A$
is (residually) hereditarily in~$\CC$.
\end{proof}

\begin{cor}\label{C_5X31_CCG2CP}
Let $\CC$ be an upwards directed class of C*-algebras.
Let $G$ be a finite abelian $2$-group,
and
let $\af \colon G \to \Aut (A)$
be an arbitrary action of $G$ on a C*-algebra~$A$.
\begin{enumerate}
\item\label{C_5X31_CCG2CP_Her}
If $A$ is hereditarily in~$\CC$,
then $C^* (G, A, \alpha)$ and $A^{\af}$
are hereditarily in~$\CC$.
\item\label{C_5X31_CCG2CP_Res}
If $A$ is residually hereditarily in~$\CC$,
then $C^* (G, A, \alpha)$ and $A^{\af}$
are residually hereditarily in~$\CC$.
\end{enumerate}
\end{cor}

\begin{proof}
For $C^* (G, A, \alpha)$,
apply Theorem~\ref{T_5X30_CCG}
with $C^* (G, A, \alpha)$ in place of $A$
and the dual action ${\widehat{\af}}$ in place of~$\af$.

For $A^{\af}$,
use the Proposition in~\cite{Rs}
to see that $A^{\af}$
is isomorphic to a corner of $C^* (G, A, \af)$,
and apply Proposition 5.10 of~\cite{PsnPh2}.
\end{proof}

Presumably Corollary~\ref{C_5X31_CCG2CP}
is valid for crossed products by coactions
of not necessarily abelian $2$-groups.
Indeed, possibly the appropriate context is that of
actions of finite dimensional
Hopf C*-algebras.
We will not pursue this direction here.

\begin{cor}\label{C_5Z28_ConseqFP}
Let $G$ be a finite $2$-group,
and
let $\af \colon G \to \Aut (A)$
be an arbitrary action of $G$ on a C*-algebra~$A$.
Suppose $A^{\af}$ has one of the following properties:
residual hereditary infiniteness,
residual hereditary proper infiniteness,
residual (SP),
or the combination of the ideal property and pure infiniteness.
Then $A$ has the same property.
\end{cor}

\begin{proof}
As discussed in the introduction,
for each of these properties there is an upwards directed
class $\CC$ such that a C*-algebra has the property
\ifo{} it is residually hereditarily in the class~$\CC$.
Apply Theorem~\ref{T_5X30_CCG}.
\end{proof}

\begin{cor}\label{C_5Z28_Conseq}
Let $G$ be a finite abelian $2$-group,
and
let $\af \colon G \to \Aut (A)$
be an arbitrary action of $G$ on a C*-algebra~$A$.
Suppose $A$ has one of the following properties:
residual hereditary infiniteness,
residual hereditary proper infiniteness,
residual (SP),
or the combination of the ideal property and pure infiniteness.
Then $C^* (G, A, \alpha)$ and $A^{\af}$ have the same property.
\end{cor}

\begin{proof}
The proof is the same as that of
Corollary~\ref{C_5Z28_ConseqFP},
using Corollary~\ref{C_5X31_CCG2CP}
instead of Theorem~\ref{T_5X30_CCG}.
\end{proof}

We omit the weak ideal property in
Corollary~\ref{C_5Z28_ConseqFP}
and Corollary~\ref{C_5Z28_Conseq},
because better results are already known
(Theorem 8.9 and Corollary 8.10 of~\cite{PsnPh2}).
We also already know
(Theorem 3.17 of~\cite{PsnPh1})
that topological dimension zero is preserved by crossed
products by actions of arbitrary finite abelian groups,
not just abelian $2$-groups.
The result analogous to Corollary~\ref{C_5Z28_ConseqFP}
is Theorem 3.14 of~\cite{PsnPh1},
but it has an extra technical hypothesis.
In the separable case,
we remove this hypothesis.

\begin{thm}\label{T_5Z20_FPAlg}
Let $\af \colon G \to \Aut (A)$
be an action of a finite group $G$
on a separable C*-algebra~$A$.
Suppose that $A^{\af}$ has topological dimension zero.
Then $A$ has topological dimension zero.
\end{thm}

\begin{proof}
Define an action $\bt \colon G \to \Aut (\OTT A)$
by $\bt_g = \id_{\OT} \otimes \af_g$ for $g \in G$.
The implication from (\ref{T_5Y28_TDz_tdZ}) to~(\ref{T_5Y28_TDz_O2wip})
in Theorem~\ref{T_5Y28_TDz}
shows that $(\OTT A)^{\bt} = \OTT A^{\af}$ has the weak ideal property.
Theorem~8.9 of~\cite{PsnPh2} now implies that
$\OTT A$ has the weak ideal property.
So $A$ has topological dimension zero
by the implication
from (\ref{T_5Y28_TDz_O2wip}) to~(\ref{T_5Y28_TDz_tdZ})
in Theorem~\ref{T_5Y28_TDz}.
\end{proof}

\section{Permanence properties for tensor products}\label{Sec_PermTP}

In this section,
we consider permanence properties for tensor products.
One of its purposes is to serve as motivation for
the results on $C_0 (X)$-algebras in Section~\ref{Sec_PermBunTD}.
The new positive result is Theorem~\ref{T_5Z21_TPtdz}:
if $A$ and $B$ are separable C*-algebras
and $A$ is exact,
then $A \otimes_{\mathrm{min}} B$
has topological dimension zero
\ifo{} $A$ and $B$ have topological dimension zero.
The exactness hypothesis is necessary
(Example~\ref{E_5Z20_NotExact}).
Still assuming this exactness hypothesis,
we also give partial results for the weak ideal property,
when one of the tensor factors
actually has the ideal property and both are separable
(Theorem~\ref{T_5Z21_TPwip}),
and when one of them has finite or Hausdorff
primitive ideal space
(Proposition~\ref{R_6203_TPS} and Proposition~\ref{C_6201_T2TP}).

The following two easy examples
show that the properties we are considering are certainly not preserved
by taking tensor products with arbitrary C*-algebras.

\begin{exa}\label{E_5Z20_01}
The algebra
$\C$ has all of topological dimension zero,
the ideal property, the weak ideal property,
and residual~(SP),
but $C ([0, 1]) \otimes \C$
has none of these.
\end{exa}

\begin{exa}\label{E_5Z20_01Inf}
The algebra $\OT$ is purely infinite and has the ideal property
but $C ([0, 1]) \otimes \OT$ does not have the ideal property.
\end{exa}

In particular,
there is no hope of any general theorem about tensor products
for properties of the form ``residually hereditarily in~$\CC$''
when only one tensor factor has the property.
Permanence theorems will therefore
have to assume that both factors have the property in question.
The following example shows that we will also need to assume that at
least one tensor factor is exact.

\begin{exa}\label{E_5Z20_NotExact}
We show that there are separable unital C*-algebras $A$ and $C$
(neither of which is exact)
which have topological dimension zero
and such that $A \otimes_{\mathrm{min}} C$
does not have topological dimension zero.
In fact,
$A$ and $C$ even have real rank zero,
and $C$ is simple.
We also show that there are
separable unital C*-algebras $B$ and $D$
which are purely infinite and have the ideal property,
but such that $B \otimes_{\mathrm{min}} D$
does not have the ideal property.
In fact,
$B$ and $D$ even tensorially absorb~$\OT$,
and $D$ is simple.

Since topological dimension zero and the weak ideal property
are preserved by passing to quotients,
it follows that no other tensor product of $A$ and $C$
has topological dimension zero,
and that no other tensor product of $B$ and $D$
even has the weak ideal property.

Let $A$ and $C$ be as in Theorem~2.6 of~\cite{PR0}.
As there, $A$ and $C$ are separable \uca{s}
with real rank zero,
$C$ is simple,
and $A \otimes_{\mathrm{min}} C$ does not have the ideal property.
These are the same algebras $A$ and $C$ as used in the
proof of Proposition~4.5 of~\cite{PR}.
Thus, $A \otimes_{\mathrm{min}} C$
does not have topological dimension zero
by Proposition 4.5(1) of~\cite{PR}.
Also,
$\OTT A \otimes_{\mathrm{min}} C$
does not have the ideal property by Proposition 4.5(2) of~\cite{PR}.
Thus taking
$B = \OTT A$ and $D = \OTT C$
gives algebras $B$ and $D$ with the required properties.
\end{exa}

We have several positive results,
but no answers for several obvious questions.
We recall known results,
then give the new result we can prove
(on topological dimension zero)
and our partial results for the weak ideal property.
We conclude with open questions.

In order to get
\begin{equation}\label{Eq_6201_NewProdPrim}
\Prim (A \otimes_{\mathrm{min}} B)
 \cong \Prim (A) \times \Prim (B).
\end{equation}
we will assume one of the algebras is exact and both are separable.
In Theorem~\ref{T_5Z21_TPtdz},
Theorem~\ref{T_5Z21_TPwip},
and Corollary~\ref{P_5Z26_WIPforTP},
these assumptions can be replaced by any other hypotheses
which imply a natural \hme{}
as in~(\ref{Eq_6201_NewProdPrim}).
Proposition 2.17 of~\cite{BlnKb}
gives a number of conditions
which imply this for the spaces of prime ideals in place of
the primitive ideal spaces,
but for separable C*-algebras this is the same thing.

\begin{thm}[Corollary~1.3 of~\cite{PR0}]\label{T_5Z26_TPIP}
Let $A$ and $B$ be C*-algebras with the ideal property.
Assume that $A$ is exact.
Then $A \otimes_{\mathrm{min}} B$
has the ideal property.
\end{thm}

\begin{thm}[Proposition~4.6 of~\cite{PR}]\label{T_5Z26_TPPIIP}
Let $A$ and $B$ be C*-algebras with the ideal property.
Assume that $B$ is purely infinite and $A$ is exact.
Then $A \otimes_{\mathrm{min}} B$
is purely infinite and has the ideal property.
\end{thm}

\begin{thm}\label{T_5Z21_TPtdz}
Let $A$ and $B$ be separable C*-algebras.
Assume that $A$ is exact.
Then $A \otimes_{\mathrm{min}} B$
has topological dimension zero
\ifo{} both $A$ and $B$ have topological dimension zero.
\end{thm}

\begin{proof}
By Proposition 2.17 of~\cite{BlnKb}
(see Remark 2.11 of~\cite{BlnKb} for the notation
in Proposition 2.16 of~\cite{BlnKb},
to which it refers),
the spaces of closed prime ideals satisfy
\[
{\operatorname{prime}} (A \otimes_{\mathrm{min}} B)
 \cong {\operatorname{prime}} (A) \times {\operatorname{prime}} (B),
\]
with the homeomorphism being implemented in the obvious way.
(See Proposition 2.16(iii) of~\cite{BlnKb}.)
Since $A$, $B$, and $A \otimes_{\mathrm{min}} B$ are all
separable,
Proposition 4.3.6 of~\cite{Pd1}
implies that prime ideals are primitive;
the reverse is well known.
So
\begin{equation}\label{Eq_5Z24_ProdPrim}
\Prim (A \otimes_{\mathrm{min}} B)
 \cong \Prim (A) \times \Prim (B).
\end{equation}

Assume $A$ and $B$ have topological dimension zero.
Then (see Definition~\ref{D_5Y27_TDZero311})
we need to prove that if $X$ and $Y$
are locally compact but not necessarily Hausdorff
spaces which have topological dimension zero,
then $X \times Y$ has topological dimension zero.
So let $(x, y) \in X \times Y$,
and let $W \subset X \times Y$
be an open set with $(x, y) \in W$.
By the definition of the product topology,
there are open subsets $U_0 \subset X$ and $V_0 \subset Y$
such that $x \in U_0$ and $y \in V_0$.
By the definition of topological dimension zero,
there are compact open (but not necessarily closed)
subsets $U_0 \subset X$ and $V_0 \subset Y$
such that $x \in U \subset U_0$ and $y \in V \subset V_0$.
Then $U \times V$ is a compact open subset of $X \times Y$
such that $(x, y) \in U \times V \subset W$.

Now assume $A \otimes_{\mathrm{min}} B$ has topological dimension zero.
We prove that $B$ has topological dimension zero;
the proof that $A$ has topological dimension zero is the same,
except that we don't need to know that exactness passes to quotients.
Choose a maximal ideal $I \subset A$.
Then $A / I$ is also exact,
by Proposition 7.1(ii) of~\cite{Kr2}.
Apply~(\ref{Eq_5Z24_ProdPrim})
as is and also with $A / I$ in place of~$A$,
use the formula for the \hme{}
from Proposition 2.16(iii) of~\cite{BlnKb},
and use the quotient map $A \to A/I$
and the map
$\Prim ( (A/I) \otimes_{\mathrm{min}} B)
 \to \Prim (A \otimes_{\mathrm{min}} B)$
it induces.
The outcome is that
\begin{align*}
\Prim (B)
& \cong \Prim ( (A/I) \otimes_{\mathrm{min}} B)
  \cong \{ I \} \times \Prim (B)
\\
& \subset \Prim (A) \times \Prim (B)
 \cong \Prim (A \otimes_{\mathrm{min}} B),
\end{align*}
and is a closed subset.
Combining Lemmas 3.6 and 3.8 of~\cite{PsnPh1},
we see that closed subsets of spaces with topological dimension zero
also have topological dimension zero.
\end{proof}

The first result for the weak ideal property requires some preparation.

\begin{ntn}\label{N_6104_Prim}
Let $A$ be a C*-algebra.
For an open set $U \subset \Prim (A)$,
we let $I_A (U) \subset A$ be the corresponding ideal.
Thus
\[
\Prim (I_A (U)) \cong U
\andeqn
\Prim (A / I_A (U)) \cong \Prim (A) \setminus U.
\]
\end{ntn}

\begin{lem}\label{L_6104_Nbhd}
Let $A$ be a C*-algebra,
let $U \subset \Prim (A)$ be open,
and let $p \in A / I_A (U)$
be a \pj.
Then there exist an open subset $V \subset \Prim (A)$,
a compact (but not necessarily closed) subset $L \subset \Prim (A)$,
and a \pj{} $q \in A / I_A (V)$,
such that
$V \subset L \subset U$
and the image of $q$ in $A / I_A (U)$ is equal to~$p$.
\end{lem}

\begin{proof}
For $P \in \Prim (A)$ let $\pi_P \colon A \to A / P$
be the quotient map,
and for an open subset $W \subset \Prim (A)$
let $\kp_W \colon A \to A / I_A (W)$
be the quotient map.
Choose $a \in A_{\mathrm{sa}}$ such that
$\kp_{U} (a) = p$.
Define
\[
V = \big\{ P \in \Prim (A) \colon
   \| \pi_P (a^2 - a) \| > \tfrac{1}{8} \big\}
\]
and
\[
L = \big\{ P \in \Prim (A) \colon
   \| \pi_P (a^2 - a) \| \geq \tfrac{1}{8} \big\}.
\]
We apply results in~\cite{Dx} to these sets.
These results are actually stated in terms of
functions on the space ${\widehat{A}}$
of unitary equivalence classes of irreducible representations
of~$A$,
with the topology being the inverse image of
the topology on $\Prim (A)$
under the standard surjection
${\widehat{A}} \to \Prim (A)$,
but they clearly apply to $\Prim (A)$.
It follows from Proposition 3.3.2 of~\cite{Dx} that $V$ is open,
and from Proposition 3.3.7 of~\cite{Dx}
that $L$ is compact.
Obviously $V \subset L$.
Clearly $\pi_P (a^2 - a) = 0$ for all $P \in \Prim (A) \setminus U$,
so $L \subset U$.

Lemma 3.3.6 of~\cite{Dx} implies that
$\| \kp_V (a^2 - a) \| \leq \tfrac{1}{8}$.
Therefore $\tfrac{1}{2} \not\in \spec ( \kp_V (a))$.
Thus we can define a \pj{} $q \in A / I_A (V)$
by $q = \ch_{(\frac{1}{2}, \I)} (\kp_V (a))$.
The image of $q$ in $A / I_A (U)$ is clearly equal to~$p$.
\end{proof}

\begin{lem}\label{L_6201_CptProd}
Let $X_1$ and $X_2$ be topological spaces,
let $W \subset X_1 \times X_2$ be an open subset,
let $x \in X_1$,
let $L \subset X_2$ be compact,
and suppose that $\{ x \} \times L \subset W$.
Then there exists an open set $U \subset X_1$
such that $x \in U$ and $U \times L \subset W$.
\end{lem}

We don't assume that $X_1$ and $X_2$ are Hausdorff.
In particular, $L$ need not be closed.

\begin{proof}[Proof of Lemma~\ref{L_6201_CptProd}]
For each $y \in L$,
choose open sets $V_1 (y) \subset X_1$ and $V_2 (y) \subset X_2$
such that $(x, y) \in V_1 (y) \times V_2 (y) \subset W$.
Use compactness of $L$
to choose $n \in \Nz$ and $y_1, y_2, \ldots, y_n \in L$
such that $V_2 (y_1), \, V_2 (y_2), \, \ldots, \, V_2 (y_n)$
cover~$L$.
Take $U = \bigcap_{j = 1}^n V_1 (y_j)$.
\end{proof}

\begin{thm}\label{T_5Z21_TPwip}
Let $A_1$ and $A_2$ be separable C*-algebras.
Assume that $A_1$ or $A_2$ is exact,
that $A_1$ has the ideal property,
and that $A_2$ has the weak ideal property.
Then $A_1 \otimes_{\mathrm{min}} A_2$
has the weak ideal property.
\end{thm}

In the diagram~(\ref{Eq_6202_Diagram}) in the proof below,
one should think of the subquotients as corresponding
to locally closed subsets of $\Prim (A_1) \times \Prim (A_2)$.
Thus,
the algebra in the middle of the top row
corresponds to $V_1 \times (V_2 \setminus T)$.
It contains a useful nonzero projection,
obtained as the tensor product of suitable
\pj{s} in $I_{A_1} (V_1)$ and $I_{A_2} (V_2) / I_{A_2} (T)$.
This subset isn't open,
so the algebra isn't a subalgebra of $A_1 \otimes_{\mathrm{min}} A_2$.
A main point in the proof is that,
given~$V_2$ and a nonzero \pj{}
$e_2 \in I_{A_2} (V_2) / I_{A_2} (S_2)$
(see~(\ref{Eq_6202_S2}) for the definition of $S_2$),
the sets $V_1$ and $T$ have been chosen so that there is
a \pj{} $p_2 \in I_{A_2} (V_2) / I_{A_2} (T)$
whose image is~$e_2$,
and so that the set $Y \cup [V_1 \times (V_2 \setminus T)]$
{\emph{is}} open.

We do not get a proof the the tensor product of two
algebras with the weak ideal property
again has the weak ideal property,
because we do not know how to reduce the size of~$V_2$
(to go with an analogous subset $T_1 \subset V_1$)
without changing the \pj~$e_2$.

\begin{proof}[Proof of Theorem~\ref{T_5Z21_TPwip}]
Replacing $A_2$ by $K \otimes A_2$,
we may assume that whenever
$I, J \subset A_2$ are ideals such that $I \subsetneqq J$,
then $J / I$ contains a nonzero \pj.

Define $X_j = \Prim (A_j)$ for $j = 1, 2$.
Using~\cite{BlnKb}
in the same way as in the proof of Theorem~\ref{T_5Z21_TPtdz},
we identify
\[
\Prim (A_1 \otimes_{\mathrm{min}} A_2) = X_1 \times X_2.
\]
The identification is
given by the map
from $X_1 \times X_2$ to $\Prim (A_1 \otimes_{\mathrm{min}} A_2)$
sending
$(P_1, P_2) \in X_1 \times X_2$ to the primitive ideal
obtained as the kernel of
$A_1 \otimes_{\mathrm{min}} A_2
 \to (A_1 / P_1) \otimes_{\mathrm{min}} (A_2 / P_2)$.
The lattice of ideals of $A_1 \otimes_{\mathrm{min}} A_2$
can thus be canonically identified with the lattice of open subsets
of $X_1 \times X_2$,
when this space is equipped with the product topology.
We simplify Notation~\ref{N_6104_Prim}
by writing $I_j (U)$ for $I_{A_j} (U)$
when $U \subset X_j$ is open,
and $I (W)$ for $I_{A_1 \otimes_{\mathrm{min}} A_2} (W)$
when $W \subset X_1 \times X_2$ is open.
We then get canonical isomorphisms
$I_1 (U_1) \otimes_{\mathrm{min}} I_2 (U_2) \cong I (U_1 \times U_2)$
for open subsets $U_1 \subset X_1$ and $U_2 \subset X_2$.

We need to show that if $Y, Z \subset X_1 \times X_2$
are open subsets such that $Y \subsetneqq Z$,
then $I (Z) / I (Y)$
contains a nonzero \pj.

Choose $x_1 \in X_1$ and $x_2 \in X_2$
such that $(x_1, x_2) \in Z \setminus Y$.
Choose open sets $U \subset X_1$ and $V_2 \subset X_2$
such that
\[
x_1 \in U,
\,\,\,\,\,\,
x_2 \in V_2,
\andeqn
U \times V_2 \subset Z.
\]
Define
\begin{equation}\label{Eq_6202_S2}
S_2 = \big\{ y \in V_2 \colon (x_1, y) \in Y \big\},
\end{equation}
which is an open subset of~$V_2$.
Since $A_2$ has the weak ideal property,
there is a \nzp{} $e_2 \in I_2 (V_2) / I_2 (S_2)$.
Use Lemma~\ref{L_6104_Nbhd}
to choose subsets $T \subset L \subset S_2 \subset V_2$
such that $L$ is compact,
$T$ is open,
and there is a \pj{} $p_2 \in I_2 (V_2) / I_2 (T)$
whose image in $I_2 (V_2) / I_2 (S_2)$ is equal to~$e_2$.
Use Lemma~\ref{L_6201_CptProd} to choose an open set
$V_1 \subset U$ such that $x_1 \in V_1$ and $V_1 \times L \subset Y$.
Define
$S_1 = V_1 \cap \big( X_1 \setminus {\overline{ \{ x_1 \} }} \big)$,
which is an open subset of $V_1$.
Since $A_1$ has the ideal property,
there is a \pj{} $p_1 \in I_1 (V_1)$
whose image $e_1 \in I_1 (V_1) / I_1 (S_1)$
is nonzero.

We claim that
\begin{equation}\label{Eq_6201_Disj}
Y \cap \big[ (V_1 \setminus S_1) \times (V_2 \setminus S_2) \big] = \E.
\end{equation}
The definitions of the sets involved imply that
\[
V_1 \setminus S_1 \subset {\overline{ \{ x_1 \} }}
\andeqn
V_2 \setminus S_2
 = \big\{ y \in V_2 \colon (x_1, y) \not\in Y \big\}.
\]
Therefore
\[
(V_1 \setminus S_1) \times (V_2 \setminus S_2)
 \subset {\overline{ \{ x_1 \} \times (V_2 \setminus S_2) }}
\andeqn
\big[ \{ x_1 \} \times (V_2 \setminus S_2) \big] \cap Y = \E.
\]
Since $Y$ is open,
the claim follows.

We now want to construct a commutative diagram as follows:
\begin{equation}\label{Eq_6202_Diagram}
\xymatrix{
I (V_1 \times V_2) \ar[d]_{\io} \ar[r]^{\pi \hspace*{2.5em}}
& I (V_1 \times V_2) / I (V_1 \times T)
    \ar[d]_{\ph} \ar[r]^{\hspace*{1em} \kp}
& I (V_1 \times V_2) / I (R) \ar[d]_{\ps} \\
I (Z) \ar[r]_{\sm \hspace*{2em}}
& I (Z) / I (Y) \ar[r]_{\rh \hspace*{1em}}
& I (Z) / I (Y \cup R).
}
\end{equation}
The maps $\pi$ and $\sm$ are the obvious quotient maps,
and $\io$ is the obvious inclusion,
coming from $V_1 \times V_2 \subset U \times V_2 \subset Z$.
We define $R = (V_1 \times S_2) \cup (S_1 \times V_2)$,
which is an open subset of $V_1 \times V_2$.
In particular, $R \subset Z$.
The map $\kp$ is then the quotient map arising from the
inclusion $V_1 \times T \subset V_1 \times S_2 \subset R$,
and $\rh$ is the quotient map arising from the
inclusion $Y \subset Y \cup R$.
Since $\pi$ is surjective,
the map $\ph$ is unique if it exists.
For existence,
we must show that $\Ker (\pi) \subset \Ker (\sm \circ \io)$.
This inclusion follows from
\[
\Ker (\pi) = I (V_1 \times T),
\,\,\,\,\,\,
\Ker (\sm \circ \io) = I \big( (V_1 \times V_2) \cap Y \big),
\andeqn
V_1 \times T \subset V_1 \times L \subset Y.
\]

It remains to construct~$\ps$.
Since $\kp$ is surjective,
the map $\ps$ is unique if it exists.
We claim that $\Ker (\kp) = \Ker (\rh \circ \ph)$.
Since $\pi$ is surjective,
it suffices to prove that
$\Ker (\kp \circ \pi) = \Ker (\rh \circ \ph \circ \pi)$.
We easily check that
\[
\Ker (\kp \circ \pi) = I (R)
\andeqn
\Ker (\rh \circ \ph \circ \pi)
 = I \big( (Y \cup R) \cap (V_1 \times V_2) \big).
\]
It follows from~(\ref{Eq_6201_Disj})
that $(Y \cup R) \cap (V_1 \times V_2) = R$,
proving the claim.
The claim implies not only that there is a map~$\ps$
making the right hand square commute,
but also that $\ps$ is injective.

The identification
$\Prim (A_1 \otimes_{\mathrm{min}} A_2) = X_1 \times X_2$
gives identifications
\[
I (V_1 \times V_2) / I (V_1 \times T)
 = I (V_1) \otimes_{\mathrm{min}} [I (V_2) / I (T)]
\]
and
\[
I (V_1 \times V_2) / I (R)
 = [I (V_1) / I (S_1)] \otimes_{\mathrm{min}} [I (V_2) / I (S_2)],
\]
with respect to which $\kp$ becomes the tensor product of
the quotient maps
\[
I (V_1) \to I (V_1) / I (S_1)
\andeqn
I (V_2) / I (T) \to I (V_2) / I (S_2).
\]
Define $q \in I (Z) / I (Y)$
by $q = \ph (p_1 \otimes p_2)$.
Then $q$ is a projection.
Moreover,
\[
\rh (q)
 = (\ps \circ \kp) (p_1 \otimes p_2)
 = \ps (e_1 \otimes e_2).
\]
Since $e_1 \neq 0$, $e_2 \neq 0$, and $\ps$ is injective,
it follows that $q \neq 0$.
Thus $I (Z) / I (Y)$ contains a nonzero \pj,
as desired.
\end{proof}

Using results from Section~\ref{Sec_WIPIP} below,
we can now give a case
in which the tensor product of C*-algebras
with the weak ideal property
again has this property.

\begin{cor}\label{P_5Z26_WIPforTP}
Let $A$ and $B$ be separable C*-algebras.
Assume that $A$ or $B$ is exact,
and that $A$ is
in the class~${\mathcal{W}}$ of
Theorem~\ref{T_5Z26_ClassForEq}.
If $A$ and $B$ have the weak ideal property,
then $A \otimes_{\mathrm{min}} B$ has the weak ideal property.
\end{cor}

The class~${\mathcal{W}}$ is the smallest class of separable C*-algebras
which contains the separable locally AH~algebras,
the separable LS~algebras,
the separable type~I C*-algebras,
and the separable purely infinite C*-algebras,
and is closed under finite and countable direct sums
and under minimal tensor products when one tensor factor is exact.

\begin{proof}[Proof of Corollary~\ref{P_5Z26_WIPforTP}]
By Theorem~\ref{T_5Y28_wipImpTDz},
the algebra $A$ has topological dimension zero.
Combine
Lemma~\ref{L_5Z26_PDS},
Lemma~\ref{L_5Z24_PTP},
Lemma~\ref{L_5Z26_LSinP},
Lemma \ref{L_5Z26_AHIsLS}(\ref{L_5Z26_AHIsLS_AH}),
Proposition~\ref{P_5Z24_Type1},
and
Theorem~\ref{T_5Y27_spiTDz},
to see that $A$ is in the class $\CP$
of Notation~\ref{N_5Z24_ClassP}.
Thus $A$ has the ideal property.
So $A \otimes_{\mathrm{min}} B$ has the weak ideal property
by Theorem~\ref{T_5Z21_TPwip}.
\end{proof}

Combining Proposition~\ref{P_6201_WIPForStable} below
with Theorem~\ref{T_5Z21_TPwip}
and with Theorem 8.5(6) of~\cite{PsnPh2},
one immediately sees that if
$A$ and $B$ are separable C*-algebras with the weak ideal property,
one of which is exact,
and $\Prim (A)$ is Hausdorff,
then $A \otimes_{\mathrm{min}} B$ has the weak ideal property.
A different argument allows one to prove this
without separability.
We give it here,
although it is based on material on $C_0 (X)$-algebras
in the next section.
We first consider the case in which $\Prim (A)$ is finite.

\begin{prp}\label{R_6203_TPS}
Let $A$ and $B$ be C*-algebras with the weak ideal property
such that $\Prim (A)$ is finite and $A$ or $B$ is exact.
Then $A \otimes_{\mathrm{min}} B$ has the weak ideal property.
\end{prp}

\begin{proof}
First suppose that $A$ is simple.
Using Proposition 2.17(2) and
parts (ii) and~(iv) of Proposition 2.16 of~\cite{BlnKb},
we see that $J \mapsto A \otimes_{\mathrm{min}} J$
is a one to one correspondence from the ideals of $B$
to the ideals of $A \otimes_{\mathrm{min}} B$;
moreover, if $J_1 \subset J_2 \subset B$ are ideals,
then
\[
(A \otimes_{\mathrm{min}} J_2) / (A \otimes_{\mathrm{min}} J_1)
  \cong A \otimes_{\mathrm{min}} (J_2 / J_1).
\]
Now let $L_1, L_2 \subset A \otimes_{\mathrm{min}} B$
be ideals with $L_1 \subset L_2$.
It follows that there exist ideals $J_1, J_2 \subset B$
with $J_1 \subset J_2$
such that $L_2 / L_1 \cong A \otimes_{\mathrm{min}} (J_2 / J_1)$.
There are \nzp{s}
$p_1 \in K \otimes A$ and $p_2 \in K \otimes (J_2 / J_1)$,
so $p_1 \otimes p_2$ is a \nzp{}
in
\[
[K \otimes A] \otimes_{\mathrm{min}} [K \otimes (J_2 / J_1)]
  \cong K \otimes (L_2 / L_1).
\]

We prove the general case by induction on $\card ( \Prim (A) )$.
We just did the case $\card ( \Prim (A) ) = 1$.
So let $n \in \N$
and suppose the result is known whenever $\card ( \Prim (A) ) < n$.
Suppose that $\card ( \Prim (A) ) = n$.
Choose a nontrivial ideal $I \subset A$.
By Proposition 2.17(2) and Proposition 2.16(iv) of~\cite{BlnKb},
the sequence
\[
0 \longrightarrow I \otimes_{\mathrm{min}} B
  \longrightarrow A \otimes_{\mathrm{min}} B
  \longrightarrow (A / I) \otimes_{\mathrm{min}} B
  \longrightarrow 0
\]
is exact.
The algebras $I \otimes_{\mathrm{min}} B$
and $(A / I) \otimes_{\mathrm{min}} B$
have the weak ideal property by the induction hypothesis,
so $A \otimes_{\mathrm{min}} B$ has the weak ideal property
by Theorem 8.5(5) of~\cite{PsnPh2}.
\end{proof}

Much of the proof of the following proposition
will be reused
in the proof of Proposition~\ref{P_6201_WIPForStable} below.

\begin{prp}\label{C_6201_T2TP}
Let $A$ and $B$ be C*-algebras
such that $A$ or $B$ is exact,
and $\Prim (A)$ is Hausdorff.
If $A$ and $B$ have the weak ideal property,
then $A \otimes_{\mathrm{min}} B$ has the weak ideal property.
\end{prp}

\begin{proof}
Set $X = \Prim (A)$.
We first claim that $A$ is a \ct{} $C_0 (X)$-algebra
with fiber $A_P = A / P$ for $P \in X$.
In the language of continuous fields,
this is Theorem~2.3 of~\cite{Fl}.
To get it in our language,
apply Theorem~3.3 of~\cite{Nls},
taking $\af \colon \Prim (A) \to X$
to be the identity map.
Identifying continuous $C_0 (X)$-algebras and continuous C*-bundles
as in Proposition~\ref{P_5Z31_Identify},
we use Corollary~2.8 of~\cite{KW1}
to see that
$A \otimes_{\mathrm{min}} B$ is a continuous $C_0 (X)$-algebra,
with fibers
$(A \otimes_{\mathrm{min}} B)_P = (A / P) \otimes_{\mathrm{min}} B$
for $P \in X$.

The algebra $A$ has topological dimension zero
by Theorem~\ref{T_5Y28_wipImpTDz}.
Since $X$ is Hausdorff,
it follows that $X$ is totally disconnected.
For every $P \in X$,
the quotient $A / P$ is simple because $\{ P \}$ is closed,
and has the weak ideal property by Theorem 8.5(5) of~\cite{PsnPh2}.
So the fiber
$(A \otimes_{\mathrm{min}} B)_P = (A / P) \otimes_{\mathrm{min}} B$
has the weak ideal property by
Proposition~\ref{R_6203_TPS}.
Theorem~\ref{T_5Z20_PrInCXAlg}(\ref{T_5Z20_PrInCXAlg_wip})
now implies that $A \otimes_{\mathrm{min}} B$
has the weak ideal property.
\end{proof}

\begin{qst}\label{Q_5Z22_TP_wip}
Let $A$ and $B$ be C*-algebras, with $A$ exact.
If $A$ and $B$ have the weak ideal property,
does $A \otimes_{\mathrm{min}} B$ have the weak ideal property?
\end{qst}

\begin{qst}\label{Q_5Z22_TP_rsp}
Let $A$ and $B$ be C*-algebras, with $A$ exact.
If $A$ and $B$ have residual~(SP),
does $A \otimes_{\mathrm{min}} B$ have residual~(SP)?
\end{qst}

\section{Permanence properties for bundles over totally
  disconnected spaces}\label{Sec_PermBunTD}

We now turn to section algebras of continuous fields
over totally disconnected base spaces.
We prove that if $A$ is the section algebra
of a bundle over a totally disconnected space,
and the fibers all have one of the properties residual~(SP),
topological dimension zero,
the weak ideal property,
or the combination of the ideal property and pure infiniteness,
then $A$ also has the same property.
Moreover, if $A$ has one of these properties,
so do all the fibers.

The section algebra of a continuous field
over a space which is not totally disconnected
will not have the weak ideal property
except in trivial cases,
and the same is true of the other properties
involving the existence of projections in ideals.
See Example~\ref{E_5Z20_01} and Example~\ref{E_5Z20_01Inf},
showing that this fails even for trivial continuous fields.
Accordingly,
we can't drop
the requirement that the base space be totally disconnected.
Indeed, we prove that for a continuous field
over a second countable locally compact Hausdorff space
with nonzero fibers,
if the section algebra is separable
and has one of the four properties above
then the base space must be totally disconnected.

The fact that the properties we consider
are equivalent to being
residually hereditarily in a suitable class~$\CC$
underlies some of our reasoning,
but knowing that a property has this form
does not seem to be sufficient for our results.

Following standard notation,
if $A$ is a C*-algebra then $M (A)$ is its multiplier algebra
and $Z (A)$ is its center.

\begin{dfn}\label{D_5Z20_CXAlg}
Let $X$ be a locally compact Hausdorff space.
Then a {\emph{$C_0 (X)$-algebra}}
is a C*-algebra $A$ together with a nondegenerate (see below)
\hm{} $\io \colon C_0 (X) \to Z (M (A))$.
Here $\io$ is {\emph{nondegenerate}}
if ${\ov{\io (C_0 (X)) A}} = A$.
\end{dfn}

Unlike in Definition 2.1 of~\cite{Nls},
we do not assume that $\io$ is injective.
This permits a \hsa{}
of $A$ to also be a $C_0 (X)$-algebra,
without having to replace $X$ by a closed subspace.

\begin{ntn}\label{N_5Z20_CXAlgNtn}
Let the notation be as in Definition~\ref{D_5Z20_CXAlg}.
For an open set $U \subset X$ we identify $C_0 (U)$
with the obvious ideal of $C_0 (X)$.
Then ${\ov{\io (C_0 (U)) A}}$
is an ideal in~$A$.
For $x \in X$, we define
\[
A_x = A \big/ {\ov{\io (C_0 (X \setminus \{ x \}) ) A}},
\]
and we let $\ev_x \colon A \to A_x$
be the quotient map.
For a closed subset $L \subset X$,
we define $A |_L = A / {\ov{\io (C_0 (X \setminus L)) A}}$.
We equip it with the $C_0 (L)$-algebra structure
which comes from the fact that $C_0 (X \setminus L)$
is contained in the kernel of the composition
\[
C_0 (X) \longrightarrow Z (M (A)) \longrightarrow Z (M (A |_L)).
\]
\end{ntn}

Thus $A_x = A |_{ \{ x \} }$.
Strictly speaking,
$A$ is the section algebra of a bundle
and
$A |_L$ is the section algebra of the restriction
of this bundle to~$L$,
but the abuse of notation is convenient.

\begin{lem}\label{L_5Z20_PropOfFld}
Let the notation be as in Definition~\ref{D_5Z20_CXAlg}
and Notation~\ref{N_5Z20_CXAlgNtn}.
Let $a \in A$.
Then:
\begin{enumerate}
\item\label{L_5Z20_PropOfFld_norm}
$\| a \| = \sup_{x \in X} \| \ev_x (a) \|$.
\item\label{L_5Z20_PropOfFld_Cpt}
For every $\ep > 0$,
the set
$\big\{ x \in X \colon \| \ev_x (a) \| \geq \ep \big\} \subset X$
is compact.
\item\label{L_5Z20_PropOfFld_usc}
The function $x \mapsto \| \ev_x (a) \|$
is upper semicontinuous.
\item\label{L_5Z20_PropOfFld_mult}
For $f \in C_0 (X)$ and $x \in X$,
we have $\ev_x (\io (f) a) = f (x) \ev_x (a)$.
\end{enumerate}
\end{lem}

\begin{proof}
When $\io$ is injective,
the first three parts are Corollary~2.2 of~\cite{Nls},
and
the last part is contained in the proof of Theorem~2.3 of~\cite{Nls}.
In the general case,
let $Y \subset X$ be the closed subset such that
\[
{\operatorname{Ker}} (\io)
= \big\{ f \in C_0 (X) \colon f |_Y = 0 \big\}.
\]
Then $A$ is a $C_0 (Y)$-algebra in the obvious way.
We have $A_x = 0$ for $x \not\in Y$,
and the function $x \mapsto \| \ev_x (a) \|$
associated with the $C_0 (X)$-algebra structure
is gotten by extending the one
associated with the $C_0 (Y)$-algebra structure
to be zero on $X \setminus Y$.
The first three parts then
follow from those for the $C_0 (Y)$-algebra structure,
as does the last when $x \in Y$.
The last part is trivial for $x \in X \setminus Y$.
\end{proof}

\begin{dfn}\label{D_5Z22_CtBundle}
Let $X$ be a locally compact Hausdorff space,
and let $A$ be a $C_0 (X)$-algebra.
We say that $A$ is a {\emph{continuous $C_0 (X)$-algebra}}
if for all $a \in A$,
the map $x \mapsto \| \ev_x (a) \|$
of Lemma \ref{L_5Z20_PropOfFld}(\ref{L_5Z20_PropOfFld_usc})
is \ct.
\end{dfn}

\begin{prp}\label{P_5Z22_CtField}
Let $X$ be a locally compact Hausdorff space
and let $A$ be a C*-algebra.
Then \hm{s} $\io \colon C_0 (X) \to Z (M (A))$
which make $A$ is a continuous $C_0 (X)$-algebra
correspond bijectively to
isomorphisms of $A$
with the algebra of \ct{} sections vanishing at infinity
of a continuous field of C*-algebras over~$X$,
as in 10.4.1 of~\cite{Dx}.
\end{prp}

\begin{proof}
This is essentially contained in Theorem~2.3 of~\cite{Nls},
referring to the definitions at the end of Section~1 of~\cite{Nls}.
\end{proof}

We will also need to use results from~\cite{KW1},
so we compare definitions.

\begin{prp}\label{P_5Z31_Identify}
Let $X$ be a locally compact Hausdorff space.
\begin{enumerate}
\item\label{P_5Z31_Identify_FromBdl}
Let
$\big( X, \, (\pi_x \colon A \to A_x)_{x \in X}, \, A \big)$
be a C*-bundle in the sense of Definition~1.1 of~\cite{KW1}.
Then $A$ is a $C_0 (X)$-algebra,
with structure map $\io \colon C_0 (X) \to Z (M (A))$
determined by the product
in Definition 1.1(ii) of~\cite{KW1},
\ifo{}
for every $a \in A$
the function $x \mapsto \| \pi_x (a) \|$
is upper semicontinuous and vanishes at infinity.
\item\label{P_5Z31_Identify_FromAlg}
Let $A$ be a $C_0 (X)$-algebra.
Then $\big( X, \, (\ev_x \colon A \to A_x)_{x \in X}, \, A \big)$
is a C*-bundle in the sense of Definition~1.1 of~\cite{KW1}
which satisfies the condition
in~(\ref{P_5Z31_Identify_FromBdl}).
\item\label{P_5Z31_Identify_Cont}
In (\ref{P_5Z31_Identify_FromBdl}) and~(\ref{P_5Z31_Identify_FromAlg}),
$A$ is a continuous $C_0 (X)$-algebra
\ifo{} the corresponding C*-bundle is continuous
in the sense of Definition 1.1(iii) of~\cite{KW1}.
\end{enumerate}
\end{prp}

\begin{proof}
Theorem~2.3 of~\cite{Nls}
and the preceding discussion
gives a one to one correspondence between
$C_0 (X)$-algebras
and upper semicontinuous bundles over $X$ in the sense
of the definitions at the end of Section~1 of~\cite{Nls}.
(The version stated there is for a special case:
the structure map
of the $C_0 (X)$-algebra is required to be injective
and the fibers of the bundle are required to be nonzero
on a dense subset
of~$X$.
But the argument in~\cite{Nls} also proves the general case.
Some of the argument is also contained in Lemma~2.1 of~\cite{KW1}.)

The difference between
Definition~1.1 of~\cite{KW1}
and the definition of~\cite{Nls}
is that \cite{KW1}
omits the requirement
(condition~(ii) in~\cite{Nls})
that for
$a \in A$ and $r > 0$,
the set
$\big\{ x \in X \colon \| a (x) \| \geq r \big\}$
be compact.
It is easy to check that a function
$f \colon X \to [0, \I)$
is upper semicontinuous and vanishes at infinity
\ifo{} for every $r > 0$
the set
$\big\{ x \in X \colon f (x) \geq r \big\}$ is compact.

Part~(\ref{P_5Z31_Identify_Cont})
is now immediate from the definitions.
\end{proof}

We prove results stating that if $X$ is totally disconnected
and the fibers of a $C_0 (X)$-algebra $A$
have a particular property,
then so does~$A$.
These don't require continuity.
We will return to continuity later in this section,
when we want to prove that if a continuous $C_0 (X)$-algebra
with nonzero fibers has one of our properties, then $X$ is totally
disconnected.

\begin{lem}\label{L_5Z20_CXAlgHsa}
Let the notation be as in Definition~\ref{D_5Z20_CXAlg}
and Notation~\ref{N_5Z20_CXAlgNtn}.
Let $B \subset A$ be a \hsa.
Let $a \in A$.
Then $a \in B$ \ifo{} $\ev_x (a) \in \ev_x (B)$
for all $x \in X$.
\end{lem}

\begin{proof}
The forward implication is immediate.

For the reverse implication,
we first claim that if $f \in C_0 (X)$
satisfies $0 \leq f \leq 1$
and if $b \in B$,
then $\io (f) b \in B$.
To prove the claim,
it suffices to consider the case $b \geq 0$.
In this case,
$\io (f) b = b^{1/2} \io (f) b^{1/2}$,
and the claim follows from the fact that $B$
is also a \hsa{} in $M (A)$.

To prove the result,
it is enough to prove that for every $\ep > 0$
there is $b \in B$ such that $\| a - b \| < \ep$.
So let $\ep > 0$.
Define $K \subset X$ by
\[
K = \left\{ x \in X \colon \| \ev_x (a) \| \geq \frac{\ep}{2} \right\}.
\]
For $x \in K$ choose $c_x \in B$ such that
$\ev_x (c_x) = \ev_x (a)$,
and define $U_x \subset X$ by
\[
U_x = \left\{ y \in X \colon \| \ev_y (c_x - a) \|
   < \frac{\ep}{2} \right\}.
\]
It follows from Lemma \ref{L_5Z20_PropOfFld}(\ref{L_5Z20_PropOfFld_Cpt})
that $K$ is compact and from
Lemma \ref{L_5Z20_PropOfFld}(\ref{L_5Z20_PropOfFld_usc})
that $U_x$ is open for all $x \in K$.
Choose $x_1, x_2, \ldots, x_n \in K$
such that the sets $U_{x_1}, U_{x_2}, \ldots, U_{x_n}$ cover~$K$.
Choose \cfn{s} $f_k \colon X \to [0, 1]$ with compact
support contained in $U_{x_k}$ for $k = 1, 2, \ldots, n$,
and such that for $x \in K$ we have
$\sum_{k = 1}^n f_k (x) = 1$
and for $x \in X \setminus K$ we have $\sum_{k = 1}^n f_k (x) \leq 1$.
Define $b \in A$ by
\[
b = \sum_{k = 1}^n \io (f_k) c_{x_k}.
\]
Then $b \in B$ by the claim.
Moreover,
if $x \in K$
then,
using Lemma \ref{L_5Z20_PropOfFld}(\ref{L_5Z20_PropOfFld_mult})
at the first step,
and $\| \ev_x (c_{x_k} - a) \| < \frac{\ep}{2}$
whenever $f_k (x) \neq 0$
at the second step,
we have
\[
\| \ev_x (b - a) \|
 \leq \sum_{k = 1}^n f_k (x) \| \ev_x (c_{x_k} - a) \|
 < \frac{\ep}{2}.
\]
Define $f (x) = 1 - \sum_{k = 1}^n f_k (x)$ for $x \in X$.
For $x \in X \setminus K$,
similar reasoning gives
\begin{align*}
\| \ev_x (b - a) \|
& \leq \big\| \ev_x (b - [1 - f (x)] a) \big\| + \| f (x) \ev_x (a) \|
\\
& \leq \sum_{k = 1}^n f_k (x) \| \ev_x (c_{x_k} - a) \|
          + f (x) \| \ev_x (a) \|
\\
& \leq [1 - f (x)] \frac{\ep}{2} + f (x) \| \ev_x (a) \|
  \leq \frac{\ep}{2}.
\end{align*}
It now follows from
Lemma \ref{L_5Z20_PropOfFld}(\ref{L_5Z20_PropOfFld_norm})
that $\| b - a \| < \ep$.
This completes the proof.
\end{proof}

\begin{cor}\label{C_5Z20_HSAisCXA}
Let $X$ be a locally compact Hausdorff space,
let $A$ be a $C_0 (X)$-algebra
with structure map $\io \colon C_0 (X) \to Z (M (A))$,
and let $B \subset A$ be a \hsa.
Then there is a \hm{} $\mu \colon C_0 (X) \to Z (M (B))$
which makes $B$ a $C_0 (X)$-algebra
and such that for all $b \in B$ and $f \in C_0 (X)$
we have $\mu (f) b = \io (f) b$.
Moreover, $B_x = \ev_x (B)$
for all $x \in X$.
\end{cor}

\begin{proof}
It follows from Lemma~\ref{L_5Z20_CXAlgHsa}
that if $f \in C_0 (X)$ and $b \in B$
then $\io (f) b \in B$.
For $f \in C_0 (X)$
we define $T_f \colon B \to B$ by $T_f (b) = \io (f) b$
for $b \in B$.
It is easy to check that $(T_f, T_f)$
is a double centralizer of~$B$,
and that $f \mapsto (T_f, T_f)$
defines a \hm{} $\mu \colon C_0 (X) \to Z (M (B))$.
Nondegeneracy of $\mu$ follows from nondegeneracy of $\io$.
The relations $\mu (f) b = \io (f) b$
and $B_x = \ev_x (B)$ hold by construction.
\end{proof}

\begin{lem}\label{L_5Z20_AppId}
Let $X$ be a locally compact Hausdorff space,
let $A$ be a $C_0 (X)$-algebra
with structure map $\io \colon C_0 (X) \to Z (M (A))$,
let $F \subset A$ be a finite set,
and let $\ep > 0$.
Then there is $f \in C_{\mathrm{c}} (X)$
such that $0 \leq f \leq 1$
and $\| \io (f) a - a \| < \ep$
for all $a \in F$.
\end{lem}

\begin{proof}
Define $K \subset X$ by
\[
K = \big\{ x \in X \colon {\mbox{there is $a \in F$ such
      that $\| \ev_x (a) \| \geq \frac{\ep}{3}$}} \big\}.
\]
It follows from Lemma \ref{L_5Z20_PropOfFld}(\ref{L_5Z20_PropOfFld_Cpt})
that $K$ is compact.
Choose $f \in C_{\mathrm{c}} (X)$
such that $0 \leq f \leq 1$ and $f (x) = 1$ for all $x \in K$.

Fix $a \in F$.
Let $x \in X$.
If $x \in K$
then,
using Lemma \ref{L_5Z20_PropOfFld}(\ref{L_5Z20_PropOfFld_mult}),
$\| \ev_x ( \io (f) a - a ) \| = 0$.
Otherwise,
again using Lemma \ref{L_5Z20_PropOfFld}(\ref{L_5Z20_PropOfFld_mult}),
\[
\| \ev_x ( \io (f) a - a ) \|
 \leq f (x) \| \ev_x (a) \| + \| \ev_x (a) \|
 < \frac{\ep}{3} + \frac{\ep}{3}
 = \frac{2 \ep}{3}.
\]
Clearly
$\sup_{x \in X} \| \ev_x (\io (f) a - a) \| \leq \frac{2 \ep}{3} < \ep$.
So $\| \io (f) a - a \| < \ep$
by Lemma \ref{L_5Z20_PropOfFld}(\ref{L_5Z20_PropOfFld_norm}).
\end{proof}

\begin{lem}\label{L_5Z20_AppIdAtx}
Let $X$ be a locally compact Hausdorff space,
let $A$ be a $C_0 (X)$-algebra
with structure map $\io \colon C_0 (X) \to Z (M (A))$,
let $z \in X$,
let $F \subset {\operatorname{Ker}} (\ev_z)$ be a finite set,
and let $\ep > 0$.
Then there is $f \in C_{\mathrm{c}} (X \setminus \{ z \})$
such that $0 \leq f \leq 1$
and $\| \io (f) a - a \| < \ep$
for all $a \in F$.
\end{lem}

\begin{proof}
The proof is essentially the same as that of Lemma~\ref{L_5Z20_AppId}.
We define $K$ as there,
observe that $z \not\in K$,
and require that $\supp (f)$,
in addition to being compact,
be contained in $X \setminus \{ z \}$.
\end{proof}

\begin{lem}\label{L_5Z20_QuotCX}
Let $X$ be a locally compact Hausdorff space,
let $A$ be a $C_0 (X)$-algebra
with structure map $\io \colon C_0 (X) \to Z (M (A))$,
and let $I \subset A$ be an ideal.
Let $\pi \colon A \to A / I$
be the quotient map.
Then there is a \hm{} $\mu \colon C_0 (X) \to Z (M (A / I))$
which makes $A / I$
a $C_0 (X)$-algebra
and such that for all $a \in A$ and $f \in C_0 (X)$
we have $\mu (f) \pi (a) = \pi (\io (f) a)$.
Moreover,
giving $I$ the $C_0 (X)$-algebra
from Corollary~\ref{C_5Z20_HSAisCXA},
for every $x \in X$ we have $(A / I)_x \cong A_x / I_x$.
\end{lem}

\begin{proof}
Let ${\overline{\pi}} \colon M (A) \to M (A/I)$
be the map on multiplier algebras induced by $\pi \colon A \to A / I$.
Define $\mu = {\overline{\pi}} \circ \io$.
All required properties of $\mu$ are obvious except for nondegeneracy.

To prove nondegeneracy,
let $b \in A / I$ and let $\ep > 0$.
Choose $a \in A$ such that $\pi (a) = b$.
Use Lemma~\ref{L_5Z20_AppId}
to choose $f \in C_{\mathrm{c}} (X)$
such that $0 \leq f \leq 1$
and $\| \io (f) a - a \| < \ep$.
Then
\[
\| \mu (f) b - b \|
 = \big\| \pi \big( \io (f) a - a \big) \big\|
 < \ep.
\]
This completes the proof of nondegeneracy.

It remains to prove the last statement.
Let $x \in X$.
Let $\ev_x \colon A \to A_x$
be as in Notation~\ref{N_5Z20_CXAlgNtn},
and let ${\ov{\ev}}_x \colon A / I \to (A / I)_x$
be the corresponding map with $A / I$ in place of~$A$.
Also let $\pi_x \colon A_x \to A_x / I_x$ be the quotient map.
Then $\pi_x \circ \ev_x$ and ${\ov{\ev}}_x \circ \pi$
are surjective,
so it suffices to show that they have the same kernel.

Let $a \in A$.
Suppose first $(\pi_x \circ \ev_x) (a) = 0$.
Let $\ep > 0$;
we prove that $\| ({\ov{\ev}}_x \circ \pi) (a) \| < \ep$.
We have $\ev_x (a) \in I_x$.
So there is $b \in I$ such that $\ev_x (b) = \ev_x (a)$.
Then $\ev_x (a - b) = 0$.
So Lemma~\ref{L_5Z20_AppIdAtx} provides
$f \in C_{\mathrm{c}} (X \setminus \{ x \})$
such that $0 \leq f \leq 1$
and $\| \io (f) (a - b) - (a - b) \| < \ep$.
By Corollary~\ref{C_5Z20_HSAisCXA},
we have $\io (f) b \in I$.
So $\pi (b) = \pi (\io (f) b) = 0$.
Thus
\[
\| \mu (f) \pi (a) - \pi (a) \|
 = \big\| \pi \big( \io (f) (a - b) - (a - b) \big) \big\|
 < \ep.
\]
Since ${\ov{\ev}}_x ( \mu (f) \pi (a) ) = 0$,
it follows that $\| ({\ov{\ev}}_x \circ \pi) (a) \| < \ep$.

Now assume that $({\ov{\ev}}_x \circ \pi) (a) = 0$.
Let $\ep > 0$;
we prove that $\| (\pi_x \circ \ev_x) (a) \| < \ep$.
Apply Lemma~\ref{L_5Z20_AppIdAtx} to the $C_0 (X)$-algebra $A / I$,
getting $f \in C_{\mathrm{c}} (X \setminus \{ x \})$
such that $0 \leq f \leq 1$
and $\| \mu (f) \pi (a) - \pi (a) \| < \ep$.
Thus $\| \pi (\io (f) a - a) \| < \ep$.
Choose $b \in I$ such that
$\big\| [\io (f) a - a ] - b \big\| < \ep$.
It follows that
\[
\big\| (\pi_x \circ \ev_x) \big( \io (f) a - a - b \big) \big\| < \ep.
\]
Since $\ev_x (\io (f) a) = 0$
and $(\pi_x \circ \ev_x) (b) = 0$,
it follows that $\| (\pi_x \circ \ev_x) (a) \| < \ep$,
as desired.
\end{proof}

The following result is closely related to~\cite{KW1}.

\begin{lem}\label{L_5Z20_TPCX}
Let $X$ be a locally compact Hausdorff space,
let $A$ be a $C_0 (X)$-algebra
with structure map $\io \colon C_0 (X) \to Z (M (A))$,
and let $D$ be a C*-algebra.
Then there is a \hm{}
$\mu \colon C_0 (X) \to Z (M (D \otimes_{\mathrm{max}} A))$
which makes $D \otimes_{\mathrm{max}} A$
a $C_0 (X)$-algebra
and such that for all $a \in A$, $d \in D$, and $f \in C_0 (X)$
we have $\mu (f) (d \otimes a) = d \otimes \io (f) a$.
Moreover,
for every $x \in X$ we have
$(D \otimes_{\mathrm{max}} A)_x \cong D \otimes_{\mathrm{max}} A_x$.
\end{lem}

\begin{proof}
The family
\[
\big( X,
 (\id_D \otimes_{\mathrm{max}} \pi_x \colon A \to A_x)_{x \in X},
 D \otimes_{\mathrm{max}}A \big)
\]
is a C*-bundle
in the sense of Definition~1.1 of~\cite{KW1}.
(See (2) on page~678 of~\cite{KW1}.)

Using exactness of the maximal tensor product
and Lemma~\ref{L_5Z20_AppIdAtx},
one verifies the hypothesis of Lemma~2.3 of~\cite{KW1}.
This lemma therefore implies that
for $b \in D \otimes_{\mathrm{max}} A$
the function $x \mapsto \| \ev_x (b) \|$
is upper semicontinuous.
It is clear that
for $d \in D$ and $a \in A$
the function $x \mapsto \| \ev_x (d \otimes a) \|$
vanishes at infinity,
and it then follows from density
that for all $b \in D \otimes_{\mathrm{max}} A$
the function $x \mapsto \| \ev_x (b) \|$
vanishes at infinity.
Now apply
Proposition~\ref{P_5Z31_Identify}.
\end{proof}

\begin{lem}\label{L_5Z20_IfDcPj}
Let $X$ be a totally disconnected locally compact Hausdorff space,
let $A$ be a $C_0 (X)$-algebra
with structure map $\io \colon C_0 (X) \to Z (M (A))$,
and let $x \in X$.
\begin{enumerate}
\item\label{L_5Z20_IfDcPj_Pj}
Let $p \in A_x$ be a \pj.
Then there is a \pj{} $e \in A$
such that $\ev_x (e) = p$.
\item\label{L_5Z20_IfDcPj_Inf}
Let $p \in A_x$ be an infinite \pj.
Then there is an infinite \pj{} $e \in A$
such that $\ev_x (e) = p$.
\end{enumerate}
\end{lem}

The use of semiprojectivity is slightly indirect,
because we don't know that there is a countable neighborhood base
at~$x$.

\begin{proof}[Proof of Lemma~\ref{L_5Z20_IfDcPj}]
We prove~(\ref{L_5Z20_IfDcPj_Pj}).
Since $\C$ is semiprojective,
there is $\ep > 0$ such that whenever $B$ and $C$ are C*-algebras,
$\ph \colon B \to C$ is a \hm,
and $b \in B$ satisfies $\| b^* - b \| < \ep$,
$\| b^2 - b \| < \ep$,
and $\ph (b)$ is a \pj,
then there exists a projection $e \in B$
such that $\ph (e) = \ph (b)$.
Since $\ev_x$ is surjective,
there is $a \in A$ such that $\ev_x (a) = p$.
By Lemma \ref{L_5Z20_PropOfFld}(\ref{L_5Z20_PropOfFld_usc}),
there is an open set $U \subset X$ with $x \in U$
such that for all $y \in U$
we have $\| \ev_y (a^* - a) \| < \frac{\ep}{2}$
and $\| \ev_y (a^2 - a) \| < \frac{\ep}{2}$.
Since $X$ is totally disconnected,
there is a compact open set $K \subset X$
such that $x \in K \subset U$.
Define $b = \io (\ch_K) a$.
Using Lemma \ref{L_5Z20_PropOfFld}(\ref{L_5Z20_PropOfFld_mult}),
we get
\[
\| \ev_y (b^* - b) \| < \frac{\ep}{2}
\andeqn
\| \ev_y (b^2 - b) \| < \frac{\ep}{2}
\]
when $y \in K$,
and $\ev_y (b^* - b) = \ev_y (b^2 - b) = 0$
when $y \in X \setminus K$.
It follows from
Lemma \ref{L_5Z20_PropOfFld}(\ref{L_5Z20_PropOfFld_norm})
that
\[
\| b^* - b \| \leq \frac{\ep}{2} < \ep
\andeqn
\| b^2 - b \| \leq \frac{\ep}{2} < \ep.
\]
Now obtain $e$ by using the choice of $\ep$
with $B = A$ and $C = A_x$.

We describe the changes needed
for the proof of (\ref{L_5Z20_IfDcPj_Inf}).
Let $T$ be the Toeplitz algebra,
generated by an isometry $s$
(so $s^* s = 1$ but $s s^* \neq 1$).
By hypothesis,
there is a \hm{} $\ph_0 \colon T \to A_x$
such that $\ph_0 (1) = p$
and $\ph_0 (1 - s s^*) \neq 0$.
Since $T$ is semiprojective,
an argument similar to that in the proof of (\ref{L_5Z20_IfDcPj_Inf})
shows that there is a \hm{}
$\ph \colon T \to A$
such that $\ev_x \circ \ph = \ph_0$.
Set $e = \ph (1)$.
Then $\ph (s)^* \ph (s) = e$
and $\ph (s) \ph (s)^* \leq e$.
We have $e - \ph (s) \ph (s)^* \neq 0$
because $\ev_x (e - \ph (s) \ph (s)^*) \neq 0$.
So $e$ is an infinite projection.
\end{proof}

\begin{thm}\label{T_5Z20_PrInCXAlg}
Let $X$ be a totally disconnected locally compact Hausdorff space
and let $A$ be a $C_0 (X)$-algebra.
\begin{enumerate}
\item\label{T_5Z20_PrInCXAlg_rsp}
Assume that $A_x$ has residual~(SP) for all $x \in X$.
Then $A$ has residual~(SP).
\item\label{T_5Z20_PrInCXAlg_piip}
Assume that $A_x$ is purely infinite and has the ideal property
for all $x \in X$.
Then $A$ is purely infinite and has the ideal property.
\item\label{T_5Z20_PrInCXAlg_wip}
Assume that $A_x$ has the weak ideal property for all $x \in X$.
Then $A$ has the weak ideal property.
\item\label{T_5Z20_PrInCXAlg_tdz}
Assume that $A$ is separable
and $A_x$ has topological dimension zero for all $x \in X$.
Then $A$ has topological dimension zero.
\end{enumerate}
\end{thm}

\begin{proof}
We prove~(\ref{T_5Z20_PrInCXAlg_rsp}).
Recall (Definition~7.1 of~\cite{PsnPh2})
that a C*-algebra~$D$ has residual~(SP)
\ifo{} $D$ is residually hereditarily in
the class $\CC$ of all C*-algebras which
contain a nonzero \pj.
(See~(\ref{Item_5X31_RSP}) in the introduction.)

We verify the definition directly.
So let $I \subset A$ be an ideal
such that $A / I \neq 0$,
and let $B \subset A / I$ be a nonzero \hsa.
Combining Lemma~\ref{L_5Z20_QuotCX}
and Corollary~\ref{C_5Z20_HSAisCXA},
we see that $B$ is a $C_0 (X)$-algebra.
Since $B \neq 0$,
Lemma \ref{L_5Z20_PropOfFld}(\ref{L_5Z20_PropOfFld_norm})
provides $x \in X$ such that $B_x \neq 0$.
Let ${\ov{\ev}}_x \colon A / I \to (A / I)_x$
be the map of Notation~\ref{N_5Z20_CXAlgNtn}
for the $C_0 (X)$-algebra $A / I$.
Then $B_x = {\ov{\ev}}_x (B)$ by Corollary~\ref{C_5Z20_HSAisCXA}
and $(A / I)_x \cong A_x / I_x$ by Lemma~\ref{L_5Z20_QuotCX}.
Thus $B_x$ is isomorphic to a nonzero \hsa{}
of $A_x / I_x$.
Since $A_x$ has residual~(SP),
it follows that there is a \nzp{} $p \in B_x$.
Lemma \ref{L_5Z20_IfDcPj}(\ref{L_5Z20_IfDcPj_Pj})
provides a \pj{} $e \in B$
such that ${\ov{\ev}}_x (e) = p$.
Then $e \neq 0$ since ${\ov{\ev}}_x (e) \neq 0$.
We have thus verified that $A$ has residual~(SP).

We next prove~(\ref{T_5Z20_PrInCXAlg_piip}).
Let $\CC$ be the class
of all C*-algebras which contain an infinite \pj.
By the equivalence of conditions (ii) and~(iv)
of Proposition~2.11 of~\cite{PR}
(valid, as shown there, even when $A$ is not separable),
a C*-algebra~$D$ is purely infinite and has the ideal property
\ifo{} $D$ is residually hereditarily in~$\CC$.
(See~(\ref{Item_5X31_PIIP}) in the introduction.)
The argument is now the same as for~(\ref{T_5Z20_PrInCXAlg_rsp}),
except using Lemma \ref{L_5Z20_IfDcPj}(\ref{L_5Z20_IfDcPj_Inf})
in place of Lemma \ref{L_5Z20_IfDcPj}(\ref{L_5Z20_IfDcPj_Pj}).

Now we prove~(\ref{T_5Z20_PrInCXAlg_wip}).
Let $\CC$ be the class
of all C*-algebras $B$ such that $K \otimes B$
contains a nonzero \pj.
It is shown at the beginning of the proof of
Theorem~8.5 of~\cite{PsnPh2}
that a C*-algebra~$D$ has the weak ideal property
\ifo{} $D$ is residually hereditarily in~$\CC$.
(See~(\ref{Item_5X31_WIP}) in the introduction.)

We verify that $A$ satisfies this condition.
So let $I \subset A$ be an ideal
such that $A / I \neq 0$,
and let $B \subset A / I$ be a nonzero \hsa.
As in the proof of~(\ref{T_5Z20_PrInCXAlg_rsp}),
$B$ is a $C_0 (X)$-algebra
and there is $x \in X$ such that $B_x$
is isomorphic to a nonzero \hsa{}
of $A_x / I_x$.
Therefore $K \otimes B_x$ contains a nonzero \pj~$p$.
Since $K$ is nuclear,
Lemma~\ref{L_5Z20_TPCX} implies that
$K \otimes B$ is a $C_0 (X)$-algebra
with $(K \otimes B)_x \cong K \otimes B_x$.
So Lemma \ref{L_5Z20_IfDcPj}(\ref{L_5Z20_IfDcPj_Pj})
provides a \pj{} $e \in K \otimes B$
such that ${\ov{\ev}}_x (e) = p$.
Then $e \neq 0$ since ${\ov{\ev}}_x (e) \neq 0$.
This shows that $A$ is residually hereditarily in~$\CC$,
as desired.

Finally,
we prove~(\ref{T_5Z20_PrInCXAlg_tdz}).
Since $A$ is separable,
by
Theorem~\ref{T_5Y28_TDz} it suffices to show that
$A$ is residually hereditarily in
the class $\CC$ of all C*-algebras $D$ such that $\OT \otimes D$
contains a nonzero \pj.
Also, for every $x \in X$,
the algebra $A_x$ is separable.
So Theorem~\ref{T_5Y28_TDz} implies that $A_x$
is residually hereditarily in~$\CC$.
The proof is now the same as for~(\ref{T_5Z20_PrInCXAlg_wip}),
except using $\OT$ in place of~$K$.
\end{proof}

We will next show that
when the $C_0 (X)$-algebra is \ct,
the fibers are all nonzero,
and the algebra is separable,
then the algebra has one of our properties
\ifo{} all the fibers have this property
and $X$ is totally disconnected.

Separability should not be necessary.

Having nonzero fibers is necessary.
The zero C*-algebra is a $C_0 (X)$-algebra for any~$X$,
and it certainly has all our properties.
For a less trivial example,
let $X_0$ be the Cantor set,
take $X = X_0 \amalg [0, 1]$,
and make $C (X_0, \OT)$ a $C (X)$-algebra
via restriction of functions in $C (X)$ to~$X_0$.

Continuity is necessary, at least without separability.

\begin{exa}\label{E_6307_Product}
Let $X$ be any \chs.
Let $A$ be the C*-algebra product $A = \prod_{x \in X} \OT$,
consisting of elements $a$ in the set theoretic product
such that $\sup_{x \in X} \| a_x \|$ is finite.
Define a homomorphism $\io \colon C (X) \to A$
by $\io (f) = (f (x) \cdot 1)_{x \in X}$ for $f \in C (X)$.
We show that this homomorphism makes $A$ a $C (X)$-algebra
which is purely infinite and has the ideal property,
the weak ideal property,
topological dimension zero,
and residual~(SP).
Taking $X = [0, 1]$
shows that without at least one of
continuity of the $C (X)$-algebra structure
and separability,
all four parts of Theorem~\ref{T_5Z20_PrInCXAlg} fail.

It is immediate to check that if $B_x$ is an
arbitrary nonzero \uca{} for $x \in X$,
then the map $\io \colon C (X) \to \prod_{x \in X} B_x$
given by $\io (f) = (f (x) \cdot 1_{B_x})_{x \in X}$ for $f \in C (X)$
is a homomorphism
which makes $\prod_{x \in X} B_x$ a $C (X)$-algebra.

It is easy to check,
by working independently in the factors in the product,
that if $B_x$ has real rank zero for all $x \in X$
then $\prod_{x \in X} B_x$ has real rank zero,
and that if $B_x$ is purely infinite and simple for all $x \in X$
then every nonzero projection in $\prod_{x \in X} B_x$ is
properly infinite.
In particular,
in the case of $A = \prod_{x \in X} \OT$,
condition (iii) in Proposition 2.11 of~\cite{PR}
is satisfied
(every hereditary subalgebra of $A$ is generated as an ideal
by its properly infinite projections).
The implications from this condition to the others
in Proposition 2.11 of~\cite{PR} do not require separability.
Thus, using condition (iv) there,
we see that $A$ has residual~(SP).
Using condition (ii) there,
we see that $A$ is purely infinite and has the ideal property.
The weak ideal property is then immediate,
and topological dimension zero
now follows from Theorem~\ref{T_5Y28_wipImpTDz}.
\end{exa}

The construction is easily adapted to spaces $X$
which are only locally compact and possibly nonunital fibers.

It is possible that requiring separability
and that all fibers be nonzero
will force $X$ to be totally disconnected.

\begin{lem}\label{L_5Z22_tdz}
Let $X$ be a second countable locally \chs,
and let $A$ be a separable continuous $C_0 (X)$-algebra
such that $A_x \neq 0$ for all $x \in X$.
If $A$ has topological dimension zero
then $X$ is totally disconnected.
\end{lem}

We assume that $X$ is second countable
because we need $A$ to be separable in Theorem~\ref{T_5Y27_spiTDz}.

\begin{proof}[Proof of Lemma~\ref{L_5Z22_tdz}]
We identify continuous $C_0 (X)$-algebras and continuous C*-bundles
as in Proposition~\ref{P_5Z31_Identify}.
Now use Corollary~2.8 of~\cite{KW1}
to see that
$\OTT A$ is a continuous $C_0 (X)$-algebra.
It follows from Theorem~\ref{T_5Y27_spiTDz}
that $\OTT A$ has the ideal property,
and then from Theorem~2.1 of~\cite{Psn6}
that $X$ is totally disconnected.
\end{proof}

\begin{thm}\label{T_5Z22_EqConds}
Let $X$ be a second countable locally \chs,
and let $A$ be a separable continuous $C_0 (X)$-algebra
such that $A_x \neq 0$ for all $x \in X$.
\begin{enumerate}
\item\label{T_5Z22_EqConds_rsp}
$A$ has residual~(SP) \ifo{}
$X$ is totally disconnected
and $A_x$ has residual~(SP) for all $x \in X$.
\item\label{T_5Z22_EqConds_piip}
$A$ is purely infinite and has the ideal property \ifo{}
$X$ is totally disconnected
and $A_x$ is purely infinite
and has the ideal property for all $x \in X$.
\item\label{T_5Z22_EqConds_wip}
$A$ has the weak ideal property \ifo{}
$X$ is totally disconnected
and $A_x$ has the weak ideal property for all $x \in X$.
\item\label{T_5Z22_EqConds_tdz}
$A$ has topological dimension zero \ifo{}
$X$ is totally disconnected
and $A_x$ has topological dimension zero for all $x \in X$.
\end{enumerate}
\end{thm}

\begin{proof}
In all four parts,
the reverse implications follow from Theorem~\ref{T_5Z20_PrInCXAlg}.
Also,
in all four parts,
the fact that $A_x$ has the appropriate property
for all $x \in X$
follows from the general fact that the property passes to
arbitrary quotients.
See Theorem 7.4(7) of~\cite{PsnPh2} for residual~(SP),
Theorem 6.8(7) of~\cite{PsnPh2} for the combination of
purely infiniteness and the ideal property,
Theorem 8.5(5) of~\cite{PsnPh2} for the weak ideal property,
and combine Proposition 5.8 of~\cite{PsnPh2}
with Theorem \ref{T_5Y28_TDz}(\ref{T_5Y28_TDz_RHO2})
for the weak ideal property.

It remains to show that all four properties imply that $X$ is totally
disconnected.
All four properties imply topological dimension zero
(using Theorem~\ref{T_5Y28_wipImpTDz} as necessary),
so this follows from Lemma~\ref{L_5Z22_tdz}.
\end{proof}

The proofs in this section depend on properties of projections,
and so do not work for a general property defined by
being residually hereditarily in an upwards directed class
of C*-algebras.
However, we know of no counterexamples
to either version of the following question.

\begin{qst}\label{Q_5Z22_HeredBundle}
Let $\CC$ be an
upwards directed class of C*-algebras,
let $X$ be a totally disconnected locally compact space,
and let $A$ be a $C_0 (X)$-algebra
such that $A_x$ is residually hereditarily in~$\CC$
for all $x \in X$
Does it follow
that $A$ is residually hereditarily in~$\CC$?
What if we assume that $A$ is a \ct{} $C_0 (X)$-algebra?
\end{qst}

\section{Strong pure infiniteness for bundles}\label{Sec_PermPI}

It seems to be unknown whether $C_0 (X) \otimes A$
is purely infinite when $X$ is a locally \chs{}
and $A$ is a general purely infinite C*-algebra,
even when $A$ is additionally assumed to be simple.
(To apply Theorem~5.11 of~\cite{KR},
one also needs to know that $A$ is approximately divisible.)
Efforts to prove this by working locally on~$X$
seem to fail.
Even in cases in which they work, such methods are messy.
It therefore seems worthwhile to give the following
result,
which, given what is known already, has a simple proof.

\begin{thm}\label{T_5Z22_PIBundle}
Let $X$ be a locally \chs,
and let $A$ be a locally trivial $C_0 (X)$-algebra
whose fibers $A_x$ are strongly purely infinite
in the sense of Definition~5.1 of~\cite{KR2}.
Then $A$ is strongly purely infinite.
\end{thm}

Since $X$ is locally compact,
local triviality is equivalent to the requirement
that every point $x \in X$ have a compact neighborhood~$L$
such that,
using the $C (L)$-algebra structure on $A |_L$
from Notation~\ref{N_5Z20_CXAlgNtn}
and the obvious $C (L)$-algebra structure on $C (L, A_x)$,
these two algebras are isomorphic as $C (L)$-algebras.
We say in this case that $A |_L$ is trivial.

\begin{proof}[Proof of Theorem~\ref{T_5Z22_PIBundle}]
Let $\io \colon C_0 (X) \to Z (M (A))$
be the structure map.

We first prove the result when $X$ is compact,
by induction on the least $n \in \N$ for which there are
open sets $U_1, U_2, \ldots, U_n \subset X$
which cover $X$ and such that $A |_{{\ov{U_j}}}$ is trivial
for $j = 1, 2, \ldots, n$.
If $n = 1$,
there is a strongly purely infinite C*-algebra~$B$
such that $A \cong C (X, B)$,
and $A$ is strongly purely infinite by Corollary~5.3 of~\cite{Kr3}.
Assume the result is known for some $n \in \N$,
and suppose that
there are
open sets $U_1, U_2, \ldots, U_{n + 1} \subset X$
which cover $X$ and such that $A |_{{\ov{U_j}}}$ is trivial
for $j = 1, 2, \ldots, n + 1$.
Define $U = \bigcup_{j = 1}^n U_j$.
If $X \setminus U = \E$
then the induction hypothesis applies directly.
Otherwise,
use $X \setminus U \subset U_{n + 1}$
to choose an open set $W \subset X$
such that $X \setminus U \subset W \subset {\ov{W}} \subset U_{n + 1}$.
Define $Y = X \setminus W$
and $L = {\ov{W}}$.
Then
\[
L \cup Y = X,
\,\,\,\,\,\,
X \setminus L \subset Y,
\,\,\,\,\,\,
Y \subset U,
\andeqn
L \subset U_{n + 1}.
\]
Since $L \subset {\ov{U_{n + 1}}}$,
there is a strongly purely infinite C*-algebra~$B$
such that $A |_L \cong C (L, B)$.
By definition
(see Notation~\ref{N_5Z20_CXAlgNtn}),
there is a short exact sequence
\[
0 \longrightarrow {\ov{\io (C_0 (X \setminus L)) A}}
  \longrightarrow A
  \longrightarrow A |_L
  \longrightarrow 0.
\]
We can identify
the algebra ${\ov{\io (C_0 (X \setminus L)) A}}$ with an ideal
in $A |_Y$.
Consideration of the sets
$U_1 \cap Y, \, U_2 \cap Y, \, \ldots, \, U_n \cap Y$
shows that the induction hypothesis applies to $A |_Y$,
which is therefore strongly purely infinite.
So ${\ov{\io (C_0 (X \setminus L)) A}}$
is strongly purely infinite
by Proposition 5.11(ii) of~\cite{KR2}.
Also $A |_L$ is strongly purely infinite by Corollary~5.3 of~\cite{Kr3},
so $A$ is strongly purely infinite by Theorem~1.3 of~\cite{Kr3}.
This completes the induction step,
and the proof of the theorem when $X$ is compact.

We now prove the general case.
Let $(U_{\ld})_{\ld \in \Ld}$ be an increasing net of open subsets
of $X$ such that ${\ov{U_{\ld}}}$ is compact for all $\ld \in \Ld$
and $\bigcup_{\ld \in \Ld} U_{\ld} = X$.
For $\ld \in \Ld$, the algebra $A |_{\ov{U_{\ld}}}$
is strongly purely infinite by the case already done.
So its ideal ${\ov{\io (C_0 (U_{\ld})) A}}$
is strongly purely infinite by Proposition 5.11(ii) of~\cite{KR2}.
Using Lemma~\ref{L_5Z20_AppId},
one checks that
$A \cong \dirlim_{\ld \in \Ld} {\ov{\io (C_0 (U_{\ld})) A}}$,
so $A$ is strongly purely infinite
by Proposition 5.11(iv) of~\cite{KR2}.
\end{proof}

\begin{lem}\label{L_5Z22_SPI}
Let $A$ be a separable C*-algebra.
Then \tfae:
\begin{enumerate}
\item\label{L_5Z22_SPI_piwip}
$A$ is purely infinite and has topological dimension zero.
\item\label{L_5Z22_SPI_spiip}
$A$ is strongly purely infinite and has the ideal property.
\end{enumerate}
\end{lem}

\begin{proof}
Condition~(\ref{L_5Z22_SPI_spiip})
implies condition~(\ref{L_5Z22_SPI_piwip})
because strong pure infiniteness implies pure infiniteness
(Proposition~5.4 of~\cite{KR2}),
the ideal property implies the weak ideal property,
and the weak ideal property implies topological dimension zero
(Theorem~\ref{T_5Y28_wipImpTDz}).

Now assume~(\ref{L_5Z22_SPI_piwip}).
Then $A$ has the ideal property
by Theorem~\ref{T_5Y27_spiTDz}.
Apply Proposition 2.14 of~\cite{PR}.
\end{proof}

\begin{cor}\label{C_5Z29_SPIBundle}
Let $X$ be a locally \chs,
and let $A$ be a locally trivial $C_0 (X)$-algebra
whose fibers $A_x$ are all purely infinite,
separable, and have topological dimension zero.
Then $A$ is strongly purely infinite.
\end{cor}

\begin{proof}
Lemma~\ref{L_5Z22_SPI} implies that the
fibers are all strongly purely infinite,
so that Theorem~\ref{T_5Z22_PIBundle} applies.
\end{proof}

\section{When does the weak ideal property imply the
  ideal property?}\label{Sec_WIPIP}

\indent
The weak ideal property seems to be the property
most closely related to the ideal property
which has good behavior on passing to hereditary subalgebras,
fixed point algebras,
and extensions.
(Example~2.7 of~\cite{PsnPh1}
gives a separable unital C*-algebra~$A$
with the ideal property
and an action of $\Zqt$ on~$A$
such that the fixed point algebra does not have the
ideal property.
Example~2.8 of~\cite{PsnPh1}
gives a separable unital C*-algebra~$A$
such that $M_2 (A)$
has the ideal property
but $A$ does not have the ideal property.
Theorem~5.1 of \cite{Psn1}
gives an extension of separable C*-algebras
with the ideal property
such that the extension does not have the ideal property.)
On the other hand,
the ideal property came first,
and in some ways seems more natural.
Accordingly,
it seems interesting to find conditions
under which the weak ideal property implies the ideal property.
Our main result in this direction is Theorem~\ref{T_5Z26_ClassForEq}.
It covers, in particular,
separable locally AH~algebras
in the sense of Definition \ref{D_5Z26_AH}(\ref{D_5Z26_AH_LocAH}) below.
We also prove (Proposition~\ref{P_6201_WIPForStable})
that the weak ideal property implies the ideal property
for stable C*-algebras with Hausdorff primitive ideal space.
We give an example to show that this implication can
fail for $Z$-stable C*-algebras.

In the introduction,
we illustrated the importance of the ideal property
with several theorems in which it is a hypothesis.
We start by showing that two of these results
can otherwise fail:
Theorem~4.1 of~\cite{Ps2} (stable rank one
for AH~algebras with slow dimension growth)
in Example~\ref{E_6401_tsr},
and Theorem~3.6 of~\cite{GJLP}
(AT~structure when in addition the K-theory is torsion free)
in Example~\ref{E_6401_AT}.
In both cases,
Theorem~\ref{T_5Z26_ClassForEq} below
implies that one can replace the the ideal property
with the weak ideal property.

\begin{exa}\label{E_6401_tsr}
Let $D$ be the $2^{\infty}$~UHF algebra.
Then $C ([0, 1]^2, \, D)$ is an AH~algebra,
even in the somewhat restrictive sense of Definition~\ref{D_5Z26_AH}
of~\cite{Ps2},
which has no dimension growth.
It follows from Proposition~5.3 of~\cite{NOP}
that $C ([0, 1]^2, \, D)$ does not have stable rank one.
Thus, Theorem~4.1 of~\cite{Ps2} fails without the ideal property.
\end{exa}

\begin{exa}\label{E_6401_AT}
Let $D$ be the $3^{\infty}$~UHF algebra,
and let $X = [0, 1]^5$.
Then $C (X, D)$ is an AH~algebra
with no dimension growth.
We show that $C (X, D)$ has torsion free K-theory
and is not an AT~algebra.
Thus, Theorem~3.6 of~\cite{GJLP} fails without the ideal property.

We have $K_0 ( C (X, D) ) \cong \Z \big[ \tfrac{1}{3} \big]$
and $K_1 ( C (X, D) ) = 0$.
Thus $K_* ( C (X, D) )$ is torsion free.
Since the real projective space $\R P^2$
is a compact $2$-dimensional manifold,
there is a closed subspace $Y \subset X$
such that $Y \cong \R P^2$.
By Proposition 2.7.7 of~\cite{At},
$K^0 (\R P^{2}) \cong \Z \oplus \Z / 2 \Z$.
Therefore
\[
K_0 ( C (Y, D) )
 \cong \Z \big[ \tfrac{1}{3} \big]
     \otimes \big( \Z \oplus \Z / 2 \Z \big)
 \cong \Z \big[ \tfrac{1}{3} \big] \oplus \Z / 2 \Z.
\]
Since this group has torsion,
$C (Y, D)$ is not an AT~algebra.
Since $C (Y, D)$ is a quotient of $C (X, D)$,
it follows that $C (X, D)$ is not an AT~algebra.
\end{exa}

It is convenient to work with the following class of C*-algebras.

\begin{ntn}\label{N_5Z24_ClassP}
We denote by $\CP$ the class of all separable
C*-algebras for which topological dimension zero,
the ideal property,
and the weak ideal property
are all equivalent.
\end{ntn}

That is,
a separable C*-algebra~$A$ is in $\CP$
exactly when either $A$ has all of the properties
topological dimension zero, the ideal property,
and the weak ideal property,
or none of them.

The class $\CP$ is not particularly interesting in its own right.
(For example, all cones over nonzero C*-algebras are in~$\CP$,
because they have none of the three properties.)
However, proving results about it will make possible a result
to the effect that these properties are all equivalent for
the smallest class of separable C*-algebras
which contains the separable AH~algebras (as well as some others)
and is closed under
certain operations.

The following lemma isolates,
for convenient reference,
what we actually need to prove to show that a separable C*-algebra
is in~$\CP$.

\begin{lem}\label{L_5Z24_ImpP}
Let $A$ be a separable C*-algebra
for which topological dimension zero implies the ideal property.
Then $A \in \CP$.
\end{lem}

\begin{proof}
The ideal property implies the weak ideal property
by Proposition 8.2 of~\cite{PsnPh2}.
The weak ideal property implies topological dimension zero
by Theorem~\ref{T_5Y28_wipImpTDz}.
\end{proof}

We prove two closure property for the class~$\CP$.
What can be done here is limited by the failure of
other closure properties for the class of C*-algebras
with the ideal property.
See the introduction to this section.
(It is hopeless to try to prove results
for $\CP$ for quotients,
since the cone over every C*-algebra is in~$\CP$).

\begin{lem}\label{L_5Z26_PDS}
Let $(A_{\ld})_{\ld \in \Ld}$ be a
countable family of C*-algebras in~$\CP$.
Then $\bigoplus_{\ld \in \Ld} A_{\ld} \in \CP$.
\end{lem}

\begin{proof}
Set $A = \bigoplus_{\ld \in \Ld} A_{\ld}$.
Then $A$ is separable,
since $\Ld$ is countable
and $A_{\ld}$ is separable
for all $\ld \in \Ld$.
By Lemma~\ref{L_5Z24_ImpP},
we need to show that if $A$ has topological dimension zero
then $A$ has the ideal property.
For $\ld \in \Ld$,
the algebra $A_{\ld}$ is a quotient of~$A$,
so has topological dimension zero by Proposition~2.6
of~\cite{BP09} and Lemma~3.6 of~\cite{PsnPh1}.
Therefore $A_{\ld}$ has the ideal property by hypothesis.

It is clear that arbitrary direct sums
of C*-algebras with the ideal property
also have the ideal property,
so it follows that $A$ has the ideal property.
\end{proof}

\begin{lem}\label{L_5Z24_PTP}
Let $A$ and $B$ be C*-algebras in~$\CP$.
Assume that $A$ is exact.
Then $A \otimes_{\mathrm{min}} B \in \CP$.
\end{lem}

\begin{proof}
The algebra $A \otimes_{\mathrm{min}} B$ is separable
because $A$ and $B$ are.
By Lemma~\ref{L_5Z24_ImpP},
we need to show that
if $A \otimes_{\mathrm{min}} B$ has topological dimension zero
then $A \otimes_{\mathrm{min}} B$ has the ideal property.
Now $A$ and $B$ have topological dimension zero
by Theorem~\ref{T_5Z21_TPtdz},
so have the ideal property by hypothesis.
It now follows from Corollary 1.3 of~\cite{PR0}
that $A \otimes_{\mathrm{min}} B$ has the ideal property.
\end{proof}

We now identify a basic collection of C*-algebras
in~$\CP$.
The main point of the first class we consider is that it
contains the separable AH~algebras (as described below),
but in fact it is much larger.

Since there are conflicting definitions of AH~algebras in the
literature,
we include a definition.
Our version is more restrictive than some other versions
in the literature,
because we insist on spaces with only finitely many connected
components and direct systems with injective maps.
We don't need to use direct sums in the definition
because we don't assume that the projections involved
have constant rank.

\begin{dfn}\label{D_5Z26_AH}
Let $A$ be a C*-algebra.
\begin{enumerate}
\item\label{D_5Z26_AH_AH}
We say that $A$ is an {\emph{AH~algebra}}
if $A$ is a direct limit of a sequence $(A_n)_{n \in \Nz}$
of C*-algebras of the form $p C (X, M_k) p$
for a \chs~$X$
with finitely many connected components,
$k \in \N$,
and a \pj{} $p \in C (X, M_k)$,
all depending on~$n$,
and in which the maps $A_n \to A_{n + 1}$ are all injective.
\item\label{D_5Z26_AH_LocAH}
We say that $A$ is a {\emph{locally AH~algebra}}
if for every finite set $F \subset A$ and every $\ep > 0$,
there exist a \chs~$X$ with finitely many connected components,
$k \in \N$, a \pj{} $p \in C (X, M_k)$,
and an injective \hm{} $\ph \colon p C (X, M_k) p \to A$
such that for all $a \in F$ there is $b \in p C (X, M_k) p$
with $\| \ph (b) - a \| < \ep$.
\end{enumerate}
\end{dfn}

In particular,
AH~algebras are locally AH~algebras.

\begin{dfn}\label{D_5Z26_LS}
Let $A$ be a C*-algebra.
\begin{enumerate}
\item\label{D_5Z26_LS_Std}
We say that $A$ is {\emph{standard}}
(Definition~2.7 of~\cite{CP})
if $A$ is unital and whenever $B$ is a simple \uca{}
and
$J \subset A \otimes_{\mathrm{min}} B$
is an ideal which is generated as an ideal by its \pj{s},
then there is an ideal $I \subset A$
which is generated as an ideal by its \pj{s}
and such that $J = I \otimes_{\mathrm{min}} B$.
\item\label{D_5Z26_LS_LS}
We say that $A$ is an {\emph{LS~algebra}}
(Definition~2.13 of~\cite{CP})
if for every finite set $F \subset A$ and every $\ep > 0$,
there exist a standard C*-algebra~$D$
and an injective \hm{} $\ph \colon D \to A$
such that for all $a \in F$ there is $b \in D$
with $\| \ph (b) - a \| < \ep$.
\end{enumerate}
\end{dfn}

\begin{lem}\label{L_5Z26_AHIsLS}
Let $A$ be a C*-algebra.
\begin{enumerate}
\item\label{L_5Z26_AHIsLS_Homog}
If $A \cong p C (X, M_k) p$
for a \chs~$X$
with only finitely many connected components,
$k \in \N$,
and a \pj{} $p \in C (X, M_k)$,
then $A$ is standard.
\item\label{L_5Z26_AHIsLS_AH}
If $A$ is a locally AH~algebra,
then $A$ is an LS~algebra.
\end{enumerate}
\end{lem}

\begin{proof}
Part~(\ref{L_5Z26_AHIsLS_Homog})
is a special case of Remark 2.9(2) of~\cite{CP}.
Part~(\ref{L_5Z26_AHIsLS_AH})
is immediate from part~(\ref{L_5Z26_AHIsLS_Homog}).
\end{proof}

There are many more standard C*-algebras
than in Lemma \ref{L_5Z26_AHIsLS}(\ref{L_5Z26_AHIsLS_Homog}),
and therefore many more LS~algebras
than in Lemma \ref{L_5Z26_AHIsLS}(\ref{L_5Z26_AHIsLS_AH}).
For example,
in Definition \ref{D_5Z26_AH}(\ref{D_5Z26_AH_LocAH})
replace $p C (X, M_k) p$
by a finite direct sum of C*-algebras
of the form $p C (X, D) p$
for connected \chs{s}~$X$,
simple unital C*-algebras~$D$,
and \pj{s} $p \in C (X, D)$.
Such a C*-algebra is standard by Remark 2.9(2) of~\cite{CP},
so a direct limit of a system of such algebras
with injective maps
is an LS~algebra.
(When all the algebras $D$ which occur are exact
and the direct system is countable,
but the maps of the system are not necessarily injective,
such a direct limit is called an
exceptional GAH algebra
in~\cite{Psn4}.
See Definitions 2.9 and~2.7 there.)

\begin{lem}\label{L_5Z26_LSinP}
Let $A$ be a separable LS~algebra
(Definition \ref{D_5Z26_LS}(\ref{D_5Z26_LS_LS})).
Then $A \in \CP$.
\end{lem}

\begin{proof}
As usual,
we use Lemma~\ref{L_5Z24_ImpP}.
Assume $A$ has topological dimension zero.
By the implication from (\ref{T_5Y28_TDz_tdZ})
to~(\ref{T_5Y28_TDz_O2IP})
in Theorem~\ref{T_5Y28_TDz},
the algebra ${\mathcal{O}}_2 \otimes A$ has the ideal property.
Apply Lemma 2.11 of~\cite{CP} with $B = {\mathcal{O}}_2$
to conclude that $A$ has the ideal property.
\end{proof}

Extending the list of properties in the discussion
of type~I C*-algebras in Remark 2.12 of~\cite{Psn6}
(and using essentially the same proof as there),
we get the following longer list of equivalent conditions
on a separable type~I C*-algebra.

\begin{prp}\label{P_5Z24_Type1}
Let $A$ be a separable type~I C*-algebra.
Then \tfae:
\begin{enumerate}
\item\label{P_5Z24_Type1_tdz}
$A$ has topological dimension zero.
\item\label{P_5Z24_Type1_wip}
$A$ has the weak ideal property.
\item\label{P_5Z24_Type1_IP}
$A$ has the ideal property.
\item\label{P_5Z24_Type1_ProjP}
$A$ has the projection property
(every ideal in~$A$
has an increasing approximate identity consisting of projections;
Definition~1 of~\cite{Psn2}).
\item\label{P_5Z24_Type1_RR0}
$A$ has real rank zero.
\item\label{P_5Z24_Type1_AF}
$A$ is an AF algebra.
\end{enumerate}
\end{prp}

\begin{proof}
It is clear that every condition on the list implies the previous one.
So we need only show that
(\ref{P_5Z24_Type1_tdz}) implies~(\ref{P_5Z24_Type1_AF}).
Use Lemma~3.6 of~\cite{PsnPh1} to see that
$\Prim (A)$ has a base for its topology consisting of compact
open sets.
Then the theorem in Section~7 of~\cite{BE}
implies that $A$ is~AF.
\end{proof}

\begin{thm}\label{T_5Z26_ClassForEq}
Let ${\mathcal{W}}$ the smallest class of separable C*-algebras
which contains
the separable LS~algebras
(including the separable locally AH~algebras),
the separable type~I C*-algebras,
and the separable purely infinite C*-algebras,
and is closed under finite and countable direct sums
and under minimal tensor products when one tensor factor is exact.
Then for any C*-algebra in ${\mathcal{W}}$,
topological dimension zero, the weak ideal property,
and the ideal property
are all equivalent.
\end{thm}

\begin{proof}
Combine
Lemma~\ref{L_5Z26_PDS},
Lemma~\ref{L_5Z24_PTP},
Lemma~\ref{L_5Z26_LSinP},
Lemma \ref{L_5Z26_AHIsLS}(\ref{L_5Z26_AHIsLS_AH}),
Proposition~\ref{P_5Z24_Type1},
and
Theorem~\ref{T_5Y27_spiTDz}.
\end{proof}

\begin{prp}\label{P_6201_WIPForStable}
Let $A$ be a C*-algebra such that $\Prim (A)$ is Hausdorff.
If $A$ has the weak ideal property
then $K \otimes A$ has the ideal property.
\end{prp}

In particular,
the weak ideal property implies the ideal property
for stable C*-algebras with Hausdorff primitive ideal space.

\begin{proof}[Proof of Proposition~\ref{P_6201_WIPForStable}]
Arguing as in the proof of Proposition~\ref{C_6201_T2TP},
we see that
$K \otimes A$ is a $C_0 (\Prim (A))$-algebra,
with fibers $(K \otimes A)_P \cong K \otimes (A / P)$
for $P \in \Prim (A)$.
Moreover, $\Prim (A)$ is totally disconnected,
and for every $P \in X$,
the quotient $A / P$ is simple
and has the weak ideal property.

For $P \in \Prim (A)$,
it follows that $K \otimes (A / P)$
is simple and has a \nzp,
so has the ideal property.
This is true for all $P \in \Prim (A)$,
so $K \otimes A$ has the ideal property
by Theorem~2.1 of~\cite{Psn6}.
\end{proof}

Let $Z$ be the Jiang-Su algebra.
It is unfortunately not true that the weak ideal property
implies the ideal property for $Z$-stable C*-algebras.

\begin{exa}\label{E_5Y28_ZStabNoIP}
We give a separable C*-algebra $A$ such that
$A$ and $Z \otimes A$ have the weak ideal property
but such that neither $A$ nor $Z \otimes A$
has the ideal property.

Let $D$ be a Bunce-Deddens algebra,
and let the extension
\[
0 \longrightarrow K \otimes D
  \longrightarrow A
  \longrightarrow \C
  \longrightarrow 0
\]
be as in the proof of Theorem~5.1 of~\cite{Psn1}.
(The extension is as in the first paragraph of that proof,
using the choices suggested in the second paragraph.)
In particular,
$A$ does not have the ideal property,
and
the connecting \hm{}
$\exp \colon K_{0} (\C) \to K_{1} (K \otimes D)$ is injective.
Since $K \otimes D$ and $\C$ have the weak ideal property
(for trivial reasons),
it follows from Theorem 8.5(5) of~\cite{PsnPh2}
that $A$ has the weak ideal property.
Clearly $Z \otimes K \otimes D$ and $Z \otimes \C$
have the ideal property.
However, it is shown in the proof of Theorem~2.9 of~\cite{Psn9}
that $Z \otimes A$
does not have the ideal property.
\end{exa}

\begin{qst}\label{Q_5Z26_ASH}
Let $A$ be a separable C*-algebra
which is a direct limit of recursive subhomogeneous C*-algebras.
If $A$ has the weak ideal property,
does $A$ have the ideal property?
\end{qst}

We suspect that the answer is no,
but we don't have a counterexample.

\end{document}